\documentclass[11pt]{amsart} 
\usepackage{amsmath,amstext,amsfonts,amssymb,url} 

\newtheorem{theorem}{\bf Theorem}[section] 
\newtheorem{theo}[theorem]{\bf Theorem} 
 
\newtheorem{corol}[theorem]{\bf Corollary} 
\newtheorem{defi}[theorem]{\bf Definition} 
 
\newtheorem{lemma}[theorem]{\bf Lemma} 
\newtheorem{notation}[theorem]{\bf Notation} 
\newtheorem{prop}[theorem]{\bf Proposition} 
 
\newtheorem{remark}[theorem]{\bf Remark} 
\newtheorem*{Appprop}{\bf Proposition} 

\newcommand{\abs}[1]{\lvert#1\rvert} 
\newcommand{\bgo}{{\oplus}} 
 
\newcommand{\ctl}{\centerline} 
\newcommand{\noi}{{\noindent}}

\newcommand{\ifff}{{if and only if }} 
\newcommand{\la}{{\langle}} \newcommand{\ra}{{\rangle}} 
\newcommand{\lf}{{\lfloor}} \newcommand{\rf}{{\rfloor}} 
 
\newcommand{\nd}{{\text{ and }}} 
\newcommand{\ov}{\overline} 
\DeclareMathOperator{\perf}{perf} 
 
\DeclareMathOperator{\sgn}{sgn} 
\newcommand{\sm}{{\smallsetminus}} 
\newcommand{\stx}{\begin{smallmatrix}} \newcommand{\estx}{\end{smallmatrix}} 
 
\newcommand{\wdt}{\widetilde} 
\newcommand{\tld}{{\hbox{\!\raise-.4ex\hbox{\,$\,\wdt{}\,$\,}}}}

\newcommand{\g}{{\gamma}} 
\newcommand{\Lb}{{\Lambda}} 
\newcommand{\lb}{{\lambda}} 
\def\eps{\epsilon} 
\newcommand{\vp}{{\varepsilon}}

\newcommand{\A}{{\mathbb A}} 
\newcommand{\D}{{\mathbb D}} 
\newcommand{\E}{{\mathbb E}} 
 
\newcommand{\R}{{\mathbb R}} 
\newcommand{\Z}{{\mathbb Z}}

\newcommand{\cC}{{\mathcal C}} 
 
\newcommand{\cE}{{\mathcal E}} 
\newcommand{\cF}{{\mathcal F}} 
\newcommand{\cH}{{\mathcal H}} 
\newcommand{\cI}{{\mathcal I}}

\newcommand{\cV}{{\mathcal V}} 
\newcommand{\cR}{{\mathcal R}} 
\newcommand{\cG}{{\mathcal G}} 

\DeclareMathOperator{\Aut}{Aut} 
\DeclareMathOperator{\GL}{GL}

\DeclareMathOperator{\Gram}{Gram} 
\DeclareMathOperator{\Id}{Id} 
 
\DeclareMathOperator{\Tr}{Tr} 

\newcommand{\sno}{{\mathcal S}^n_{>0}}
\newcommand{\sn}{{\mathcal S}^n}

\newcounter{alg}  
\newenvironment{bigalg}{\medskip%
                        \begin{figure}[htbp]%
                        \refstepcounter{alg}%
                        \begin{center}}%
                       {\end{center}\end{figure}\medskip}

\begin{document}

\title{On classifying Minkowskian sublattices}  
\author[W. Keller, J. Martinet, A. Sch\"urmann] 
{Wolfgang Keller, Jacques Martinet and Achill Sch\"urmann\\ \\ {\tiny (with an Appendix by Mathieu Dutour Sikiri\'c)} \\ }



\thanks{The first and the third author were supported by the
Deutsche Forschungsgemeinschaft (DFG) under 
grant SCHU 1503/4-2. 
The third author was additionally supported by the Universit\'e Bordeaux 1.}

\address{%
Faculty of Mathematics, 
Otto-von-Guericke Universit\"at,  
39106 Magdeburg,
\linebreak\indent 
 Germany} 
\email{Wolfgang.Keller@student.uni-magdeburg.de} 

\address{%
Universit\'e de Bordeaux, 
Institut de Math\'ematiques, 
351, cours de la 
\linebreak\indent 
Lib\'e\-ration, 
33405 Talence cedex, France} 
\email{Jacques.Martinet@math.u-bordeaux1.fr} 

\address{%
Institute of Mathematics,
University of Rostock,
18051 Rostock,
Germany}
\email{achill.schuermann@uni-rostock.de} 

\address{%
Rudjer Boskovi\'c Institute,
Bijenicka 54, 
10000 Zagreb, Croatia}
\email{mdsikir@irb.hr}

\begin{abstract} 
Let $\Lb$ be a lattice in an $n$-dimensional Euclidean space $E$ 
and let $\Lb'$ be a Minkowskian sublattice of $\Lb$, that is, 
a sublattice having a basis made 
of representatives for the Minkowski successive minima of $\Lb$. 
We extend the classification of possible $\Z/d\Z$-codes 
of the quotients $\Lb/\Lb'$ to dimension~$9$,
where $d\Z$ is the annihilator of $\Lb/\Lb'$.
\end{abstract}

\maketitle

\setcounter{tocdepth}{1} 
\tableofcontents

\newpage

%
%
%

\section{Introduction}\label{secintro} 

Let $E$ be an $n$-dimensional Euclidean space, with scalar product 
$x\cdot y$. The {\em norm of $x\in E$} is $N(x)=x\cdot x$ 
(the square of the ``classical norm'' $\|x\|$). 
Let $\Lb$ be a {\em lattice in~$E$ of rank n}, 
that is, a full rank discrete subgroup of $E$ and
a $\Z$-module in $E$ of rank $n$.
Let $m_1,\dots,m_n$ be its {\em successive minima} in the sense of Minkowski: 
each $m_i$ is the smallest real number 
such that the span of the set of vectors in $\Lb$ of norm $N\le m_i$ 
is of dimension at least $i$. 
A {\em Minkowskian sublattice} 
$$
\Lb' = \langle e_1,\dots,e_n \rangle  = \Z e_1 + \ldots + \Z e_n 
$$
of $\Lb$ is one having a {\em basis} consisting of linearly independent 
representatives $e_1,\dots,e_n$ of $m_1,\dots,m_n$.
Let $d\Z$ be the annihilator of $\Lb/\Lb'$. 
Then 
$$
\Lb = \langle \Lb', f_1,\dots,f_k \rangle
$$
is generated by the vectors $e_i$ together with some 
vectors $f_1,\dots,f_k\in\Lb$ of the form 
$$f_i=\frac{a_1^{(i)}e_1+\dots+a_n^{(i)}e_n}d\, ,$$ 
where the vectors $a^{(i)} = (a_1^{(i)},\dots,a_n^{(i)})\mod d$ 
are the words of a {\em code over $\Z/d\Z$}. 
We attach in this way to $\Lb$ a collection of codes over $\Z/d\Z$ 
which depend on the choice of the~$e_i$. We consider the problem 
of classifying for a given dimension~$n$ the set of codes which arise 
for some lattice~$\Lb\in E$ (up to equivalence). 

\smallskip 

This problem was first considered by Watson in \cite{watson-1971}, 
who obtained in particular the classification for $n\le 6$. 
This theory of Watson was then extended by Ryshkov 
(see~\cite{ryshkov-1976}) to $n=7$. Zahareva (\cite{zaharova-1980}) considered 
the problem for $n=8$. Her results were completed by the second 
author in~\cite{martinet-2001}, where also new concepts, such as the perfection 
rank or the minimal class of a lattice, were introduced. 
This latter paper will be our basic reference for what follows.

\smallskip\noi 
The index theory has various applications. 
The results of this paper will help
to gain a better understanding of lattices
in dimension~$9$ and above.
For example, we shall consider 
in a forthcoming paper~\cite{ms-2010} 
the question of the existence of a basis 
of minimal vectors for lattices generated by their minimal 
vectors. Based on the classification of this paper,
it appears possible to resolve this question 
in the currently open cases of dimension~$9$ and~$10$.
Another future application may be a computer assisted
classification of perfect forms in dimension~$9$.

\smallskip 

It should be noted that the two mentioned applications make 
use only of results for {\em well rounded lattices},
that is, for lattices with minimal vectors spanning~$E$.
In other words, these lattices have equal successive minima~$m_1,\dots,m_n$. 
Indeed, a deformation argument (see~\cite[Theorem~1.5]{martinet-2001}) shows 
that all codes can be realized using well rounded lattices. 
So from now on, we shall no longer work with the successive minima. 
Moreover, since the set of codes associated with a lattice~$\Lb$ only 
depends on the similarity class of~$\Lb$, we shall in general 
work with lattices of minimum~$1$, except that the lattices 
we shall exhibit will be scaled for convenience to the smallest 
minimum which makes them integral.

\smallskip 

Any code $C$ of length~$n$ can be trivially extended 
to all dimensions $n+k$ by adding $k$~columns of zeros to a generator 
matrix for~$C$. On the side of lattices, these codes can be realized 
by convenient direct sums of both $\Lb$ and $\Lb'$ and $k$~copies 
of~$\Z$. In particular, we may consider $(\Lb\perp\Z^k,\Lb'\perp\Z^k)$.
For this reason, we shall systematically restrict ourselves 
to codes which do not extend trivially a code of smaller length, 
as was done in~\cite[Table~11.1]{martinet-2001}. 

\smallskip 

A complete list of the existing codes for $n=9$ can be found
in Sections~\ref{sec:cyclic} and~\ref{sec:noncyclic};
in all cases we give the most important invariants.
There are $137$~codes in dimension~$9$, whereas only~$42$ codes exist in dimensions
$n\leq 8$ all together.
Our results are too complex 
to be shortly described in this introduction, so that we shall content 
ourselves here with a crude result, namely the list of possible 
structures of $\Lb/\Lb'$, merely viewed as an abstract Abelian group. 
By the comments above, it suffices to list for each dimension~$n$ 
the group structures which exist in this dimension but not 
in dimension~$n-1$. We use the standard convention for 
quoting Abelian groups by their elementary divisors, 
writing for example for short $8$, $4\cdot 2$, $2^3$ 
for the groups of order~$8$ isomorphic to 
$\Z/8\Z$, $\Z/4\Z\times\Z/2\Z$ and $\Z/2\times\Z/2\Z\times\Z/2\Z$.

\begin{theo}\label{thm:structquot} 
The possible structures for quotients $\Lb/\Lb'$ as above 
up to dimension~$n=9$ are as follows: 

\noi\underbar{$n=1$}\,: $1$\,; 
\newline 
\underbar{$n=4$}\,: $2$\,; 
\newline 
\underbar{$n=6$}\,: $3,\,2^2$\,; 
\newline 
\underbar{$n=7$}\,: $4,\,2^3$\,; 
\newline 
\underbar{$n=8$}\,: $5,\,6,\,4\cdot 2,\,3^2,\,2^4$\,; 
\newline 
\underbar{$n=9$}\,: $7,\,8,\,9,\,10,\,12,\,6\cdot 2,\,4^2,\,4\cdot 2^2$\,. 
\newline 
Moreover, all structures which exist in dimension 
$n=4$, $7$, $8$ exist for the lattices 
$\D_4$, $\E_7$, $\E_8$ respectively,
but no such ``universal lattice'' exists in dimensions $6$ and $9$.
For the laminated lattice~$\Lb_9$ only the quotient $4^2$ is missing.
We refer to Appendix~A 
for more information on the mentioned lattices.
\end{theo}

The results for $n\le 8$ were obtained in~\cite{martinet-2001},  
using essentially calculations by hand.
After many codes were {\em a priori} excluded, a computer was used only
to find lattices, proving the existence for the remaining codes.
The complication of some proofs however (e.g., the non-existence 
of cyclic quotients $\Lb/\Lb'$ of order~$8$), clearly shows 
that the methods of~\cite{martinet-2001} are no longer suitable in higher 
dimensions, at least when it involves an index $[\Lb:\Lb']\ge 7$.
So here we develop a method that also allows us to prove non-existence of
codes using computer assistance. Our calculations not only 
verify all of the previously known results for $n\leq 8$, 
but also allow us to give a full classification for $n=9$.

In~\cite{zaharova-1980}, Zahareva introduces the notion of a {\em free pair} 
$(\Lb,\Lb')$: a pair of well rounded lattices such that the set of minimal vectors 
of~$\Lb$ reduces to the basis vectors $\pm e_i$ of~$\Lb'$. 
For each given structure of the Abelian group $\Lb/\Lb'$, 
there are minimal dimensions $n_0$ and $n_1$ such that $n$-dimensional 
lattices with the given structure exist for all $n\ge n_0$ and some of them are free 
for all $n\ge n_1$. Table~\ref{tab:free-pairs} shows information
on these minimum dimensions up to index~$8$ that follows from our classification.

\begin{table}[ht]
\begin{tabular}{|c|c|c|c|c|c|c|c|c|c|c|}
\hline
$[\Lb:\Lb']$  &  $2$  &  $3$  &   $4$  &  $2^2$  &  $5$  &  $6$  &   $7$  &  $8$  &  $4\cdot 2$  &  $2^3$ \\ 
\hline
\hline
exists  &  $4$  &  $6$  &   $7$  &  $6$  &  $8$  &  $8$  &   $9$  &  $9$  &  $8$  &  $7$  \\ 
free    &  $5$  &  $7$  &   $7$  &  $8$  &  $8$  &  $9$  &   $9$  &  $9$  &  $9$  &  $10$  \\
\hline
\end{tabular}
\medskip
\caption{Existence and free quotients}
\label{tab:free-pairs}
\end{table}

The remaining of the paper is organized as follows:

\smallskip\noi 
In Section~\ref{secidentities}, we set some notation and discuss 
some bounds for the index $[\Lb:\Lb']$, by which it becomes clear
that in each dimension only finitely many codes have to be considered.
We describe some identities which further allow 
us to considerably reduce the number of codes which need be 
considered. 

\smallskip\noi 
In Section~\ref{sec:quadforms}, we recall the basic dictionary 
between lattices and positive definite, real symmetric matrices.
We in particular review some facts about the Ryshkov polyhedron
that parametrizes all lattices whose non-zero vectors are of length
at least~$1$. We establish a connection between its facial
structure and the possible minimal classes of a lattice.
We show that each code over $\Z/d\Z$ of length $n$ is associated
with a unique minimal class and a unique set of faces of the
Ryshkov polyhedron.
 
\smallskip\noi 
In Section~\ref{sec:algorithm}, based on the connection established in 
Section~\ref{sec:quadforms}, we give an algorithm that allows us to 
test whether or not a given $\Z/d\Z$ code can be realized by a 
pair of lattices $(\Lb,\Lb')$.

\smallskip\noi 
In Section~\ref{sec:restrictnumberofcodes} we give criteria due to Watson
that easily allow us to exclude many codes from further considerations.

\smallskip\noi 
In Section~\ref{sec:cyclic}, we consider cyclic quotients $\Lb/\Lb'$ 
which exist in dimension $9$, but not in dimension~$8$.
We give a complete list of corresponding codes (see~Table~\ref{tab:cyclic}).
Whereas our results for $d\geq 7$ depend on computer calculations, 
we give arguments for all ``small'' cyclic quotients of order $d\leq 6$. 
We hereby establish the classification in all 
dimensions for lattices~$\Lb$ with maximal index $[\Lb:\Lb']\le 6$,
for sublattices $\Lb'$ generated by minimal vectors of $\Lb$.

\smallskip\noi 
Section~\ref{sec:noncyclic} is devoted to non-cyclic quotients. 
We consider for all dimensions, 
quotients of type $2^k$, $3^k$ and $4\cdot 2^k$.
In order to give a complete list of possible codes in dimension~$9$,
other cases are treated computationally, giving overall a computer assisted proof of Theorem~\ref{thm:structquot}.
All of the existing, $9$-dimensional codes are listed in 
Tables~\ref{tab:n9d22-table}, \ref{tab:n9d222-table}, \ref{tab:n9d2222-table}, \ref{tab:n9d33-table},
\ref{tab:n9d42-table}, \ref{tab:n9d422-table}, \ref{tab:n9d62-table}, and~\ref{tab:n9d44-table}.

\smallskip\noi 
In Section~\ref{secuniversal}, we discuss the existence 
of lattices $\Lb$ which are universal in the sense that every quotient 
$L/L'$ which exists in dimension~$n$ indeed exists for $L=\Lb$; 
see Theorem~\ref{thm:structquot}. 
It turns out that such a lattice does not exist for $n=9$.
However, all structures except quotients of type $4^2$
are attained by the lattice $\Lb_9$.

\smallskip\noi 
Appendix~A is devoted to perfect lattices that occur 
at several places in the presented ``index theory'': 
the root lattices, the laminated lattices, and in particular the Leech lattice.
Appendix~B by Mathieu Dutour Sikiri\'c describes a strategy to 
compute the index system of a given lattice. We used his computations
to check our results. It helped to discover 
problems in an earlier version of this paper.

\smallskip\noi 
In addition to the information contained in this article, extra data and {\tt MAGMA} scripts 
accompanying our classification are available as an ``online appendix''.
To access these, either download the source files for the arXiv paper {\tt arXiv:0904.3110}
or download it from the corresponding world wide web page of {\em Mathematics of Computation}.
The file {\tt Gramindex.gp} contains a Gram matrix in {\tt PARI-GP} format for every found lattice type.

\section{Bounds and identities}\label{secidentities} 

Recall that we consider pairs $(\Lb,\Lb')$ where $\Lb$ 
is a well-rounded lattice in an $n$-dimensional Euclidean 
space~$E$ and $\Lb'\subset\Lb$ is generated by $n$~independent 
minimal vectors of $\Lb$. We denote by $x\cdot y$ 
the scalar product on $E$ and define the {\em norm of $x\in E$} 
by $N(x)=x\cdot x$. 
The {\em minimum of $\Lb$} (which is actually attained) is 
$$\min\Lb=\inf_{x\in\Lb\sm\{0\}}\,N(x)\,.$$ 
The set of {\em minimal vectors of $\Lb$} is 
$$S(\Lb)=\{x\in\Lb\mid N(x)=\min\Lb\}\,,$$ 
and we define $s=s(\Lb)$ by $\abs{S(\Lb)}=2s$. The {\em Gram matrix 
of an ordered set $\cE=(x_1,\dots,x_k)$ of vectors of~$E$} 
is the $k\times k$ matrix $\Gram(\cE)=(x_j\cdot x_k)$. 
The {\em determinant $\det(\Lb)$ of $\Lb$} is the determinant 
of the Gram matrix of any basis for~$\Lb$. 
Finally, the {\em Hermite invariant of $\Lb$} 
and the {\em Hermite constant of dimension~$n$} are 
$$\g(\Lb)=\frac{\min\Lb}{\det(\Lb)^{1/n}}\ \nd\ 
\g_n=\sup_{\dim\Lb=n}\,\g(\Lb)\,.$$ 
The following result is well-known:

\begin{prop}\label{propmaxind} 
With the notation and the hypotheses above, we have 
$$[\Lb:\Lb']\le \lfloor \g_n^{n/2} \rfloor\,.$$ 
\end{prop} 

\begin{proof} 
By the definition of the determinant and the index we 
have $\det(\Lb')=[\Lb:\Lb']^2\cdot\det(\Lb)$.
Further,
$\det(\Lb')\le  N(e_1)\cdots N(e_n) = 1$ by the Hadamard inequality 
and the assumption $\min \Lb = 1$. 
(We refer to~\cite{martinet-2003} for the corresponding background.)
As $\det(\Lb)\ge \g_n^{-n}$ by definition of the Hermite constant,
and as the index is a natural number, the result follows. 
\end{proof}

\begin{defi}\label{defmaxind}{\rm 
The {\em maximal index $\imath(\Lb)$ of the well-rounded lattice $\Lb$} 
is the largest value that $[\Lb:\Lb']$ may attain when $\Lb'$ runs 
through the set of sublattices of $\Lb$ which are generated 
by $n$~independent minimal vectors of~$\Lb$. 
 
The {\em index system $\cI(\Lb)$ of $\Lb$} is the set of all structures 
of Abelian groups provided by quotients $\Lb/\Lb'$ as above.}
\end{defi} 

\smallskip

\noi
{\bf Example:} By Theorem~\ref{thm:structquot}, the index system 
of~$\Lb_9$ is 

\smallskip\ctl{$\cI(\Lb_9)=\{1,2,3,4,2^2,5,6,7,8,4\cdot 2,2^3,9,3^2,%
10,12,6\cdot 2,4\cdot 2^2,2^4\}$\,.}

\smallskip

The Hermite constants and the {\em critical lattices} 
on which they are attained are known for $n\le 8$ and $n=24$. 
For other values of $n$, we must content ourselves with upper bounds 
valid for all sphere packings. The best bounds in print are those 
of Cohn and Elkies \cite{ce-2003}. In particular, we have 
\begin{equation} \label{eqn:boundongamma9and10}
\g_9^{9/2}\le 30.21\ \nd\ \g_{10}^5\le 59.44\,.
\end{equation}
Note that the conjectural values, namely those of the laminated 
lattices $\Lb_9$, $\Lb_{10}$ (defined in \cite[Chapter~6]{cs-1998}) are 
$$
\g(\Lb_9)^{9/2}=2^{9/2}=22.627\dots\ \nd\ 
\g(\Lb_{10})^5=\frac{4^5}{\sqrt{768}}=36.950\dots\,,
$$
which give much smaller bounds for $\imath(\Lb)$ 
in these two dimensions.

\begin{remark}\label{remsmallindex} 
It results from~\cite{martinet-2001} that the bound 
$\imath(\Lb)\le\lf\g_n^{n/2}\rf$ is exact for $n\le 8$, 
and that the precise equality $\imath(\Lb)=\g_n^{n/2}$ 
even holds for $n=4,7,8$, with $\Lb$ one of the root lattices 
$\D_4$, $\E_7$, $\E_8$; 
the bound is also tight for $n=24$ with the {\em Leech lattice} $\Lb_{24}$
(see Appendix~A).
Note that by Theorem~\ref{thm:structquot}
the bound is strict for~$n=9$,
as the largest possible value for $\imath(\Lb)$ 
is then~$16$, the same as for $n=8$. 
We conjecture that even the conjectural bound $\imath(\Lb) \le 36$ 
for $n=10$ is strict, the actual bound being probably $32$. 
\end{remark}

\section{Ryshkov polyhedron and minimal classes}
\label{sec:quadforms}

In Section~\ref{sec:algorithm} we formulate an algorithm 
for determining whether or not a given code $C$  
can be realized.
For it we use the language of quadratic forms, or equivalently, 
of real symmetric matrices.
Instead of looking at bases of lattices, we
consider their positive definite Gram matrices. 
Note that there is a well known dictionary 
translating between lattice and Gram matrix terminology.
There is in particular a one-to-one correspondence between 
$n$-dimensional lattices up to orthogonal transformations
and Gram matrices $G$ up to 
the $\GL_n(\Z)$ action 
$
G \mapsto U^t G U
$.

By $\sn$ we denote 
the space of real symmetric $n\times n$ matrices.
It is turned into a Euclidean space
with the usual inner product
$\langle A,B \rangle = \Tr AB$.
For $G\in \sn$ and $x\in \R^n$ we write   
$G[x]=x^tGx$. We note that $G[x]=\langle G, xx^t\rangle$
is a linear function on $\sn$ for a fixed $x\in\R^n$.

Let $\sno$ denote the set of positive definite matrices within $\sn$.
It is well known that $\sno$ is an open convex cone whose
closure is the set of positive semi-definite matrices.
In accordance with the definition for lattices, we define
the {\em minimum of $G$} by
$$
\min G = \min_{x\in\Z^{n}\setminus\{0\}} G[x]
$$
and its set of {\em minimal vectors} by
$$
S(G)= \{ x\in\Z^{n} \; | \; G[x] = \min G \}
.
$$

Within $\sno$, the set of Gram matrices $G$ 
with minimum $\min(\Lb)$ at least~$1$ form
a locally finite polyhedron -- 
the {\em Ryshkov polyhedron} 
$$
\cR
=
\{
G \in \sn \; | \;  G[x] \geq 1 \mbox{ for all } x \in \Z^{n}\setminus\{0\}
\}
.
$$ 
Here ``locally finite'' means:
for Gram matrices $G$ in any fixed, bounded part of $\cR$,
all except finitely many of the inequalities $G[x]\geq 1$ are strict
(see~\cite{schuermann-2008} for a proof).
As a consequence, bases of lattices with minimum $1$ are identified with a
piecewise linear surface in $\sno$ (the boundary of $\cR$). 
Its {\em faces} form a cell complex, naturally carrying 
the structure of a {\em combinatorial lattice} with respect to inclusion.
Note that the relative interiors of faces are disjoint, whereas
the closed faces themselves may meet at their boundaries.
For basic terminology and results from the theory of 
polyhedra we refer to~\cite{ziegler-1997}.

The group $\GL_n(\Z)$
acts by $G\mapsto U^t G U$ on the Ryshkov polyhedron and its boundary.
All bases of a given lattice $\Lb$ with minimum $1$ 
yield Gram matrices that
lie in the relative interior of faces of the same dimension $k$.
This invariant is the {\em perfection co-rank} of $\Lb$;
its {\em perfection rank} is 
$$
\perf \Lb = \dim \sn - k
.  
$$

By a well known theorem of Voronoi (see~\cite{voronoi-1907}), we know
that up to the action of $\GL_n(\Z)$,
there exist only finitely many vertices ($0$-dimensional faces) 
of the Ryshkov polyhedron. They are called {\em perfect},  
as are the corresponding lattices, which are the lattices having 
full perfection rank $\dim \sn = \binom{n+1}{2}$.
As a consequence of Voronoi's finiteness result,
there exist only finitely many orbits of faces of any dimension.
Thus we obtain an abstract finite complex 
(quotient complex) from the face lattice of $\cR$
modulo the action of $\GL_n(\Z)$.

Under the action of $\GL_n(\Z)$, the relative interiors of 
faces of $\cR$ fall into equivalence classes. The corresponding
equivalence classes of lattices are called {\em minimal classes}.
The inclusion of faces $\cF'\subseteq\cF$ induces a (reversed) 
ordering relation on corresponding minimal classes, denoted by $\cC\prec\cC'$.
With respect to this ordering, the minimal classes form a combinatorial
lattice that is anti-isomorphic to the face lattice of the quotient complex 
described above.
Note that 
lattices $\Lb$ and $\Lb'$ in the same minimal class $\cC$ are 
characterized by the fact that there exists a transformation
\begin{equation} \label{eqn:minclass}
u\in\GL(E)
\quad \mbox{ with } \quad
u(\Lb)=\Lb' \nd u(S(\Lb)) = S(\Lb')
.
\end{equation}
Inclusion of minimal vector sets $u(S(\Lb))\subseteq S(\Lb')$ 
induces the same ordering relation $\cC\prec\cC'$ on minimal classes.

Given $\Lb'\subset\Lb$ having a basis of minimal vectors of $\Lb$, 
all lattices $L$ of the minimal class of $\Lb$ 
contain a sublattice $L'$ such that $\Lb/\Lb'$ and $L/L'$ define 
the same code over $\Z/d\Z$, where $d\Z$ is the annihilator 
of $\Lb/\Lb'$. 
This follows from the existence of a transformation $u$ as in
\eqref{eqn:minclass}. 
As a consequence, given a minimal class $\cC$, 
{\em the set of codes attached to pairs $(\Lb,\Lb')$ as above 
with $\Lb\in\cC$ is an invariant of $\cC$}; and more generally, 
the set of codes and index system of a class $\cC'\succ\cC$ 
contain those of~$\cC$. This implies that for the classification of
possible codes and index systems, it would suffice to study the finitely many 
perfect lattices of minimum~$1$. However, in dimension~$9$ these are not
fully known and it appears that there exist too many of them 
for such an approach (see \cite{dsv-2007}).

Given a $\Z/d\Z$-code $C$, we consider the set of 
minimal classes of well-rounded lattices $\Lb$ such that $\Lb/\Lb'$ 
defines the code $C$ for a suitably chosen sublattice $\Lb'$ of $\Lb$,
having a basis of minimal vectors of~$\Lb$.
As shown by the following proposition, 
all of these {\em well rounded minimal classes}
are attached to a uniquely determined {\em minimal class $\cC_C$ of $C$}.

\begin{prop} \label{prop:smallclass}
Let $C$ be a $\Z/d\Z$-code. 
Then there exists a unique well-rounded minimal class $\cC_C$,
such that $\cC(\Lb)\succ\cC_C$, for every well-rounded lattice $\Lb$
with sublattice $\Lb'$ generated by $n$~minimal vectors of~$\Lb$
such that $\Lb / \Lb'$ defines the code $C$.
\end{prop}

For the proof of the proposition, 
we give a geometric argument involving the Ryshkov polyhedron $\cR$,
which leads us to the main idea underlying the algorithm 
that we treat in the next section.
We show that there exists a uniquely determined orbit of a face
$\cF$ of the Ryshkov polyhedron $\cR$ for every code $C$ that exists.

\begin{proof}[Proof of Proposition \ref{prop:smallclass}]
Assume the code $C$ is generated by $k$~code words $a^{(i)}$, $i=1,\ldots,k$.
So we may assume the lattice $\Lb'$ has a basis of minimal 
vectors $e_1,\ldots, e_n$ of $\Lb$ and  
$\Lb = \langle \Lb', f_1,\dots,f_k \rangle$ with 
$$
f_i=\frac{a_1^{(i)}e_1+\dots+a_n^{(i)}e_n}d \, ,
$$ 
for $i=1,\ldots,k$. 
Choose a basis $B=(b_1,\ldots, b_n)$ of $\Lb$.
Then $e_i$ has coordinates $\bar{e}^{(i)}\in \Z^n$
with respect to the chosen basis $B$. Note that these 
coordinates can be expressed in terms of the $a_j^{(i)}$ and $d$,
independently of the specific lattices~$\Lb$ and~$\Lb'$.

Assuming the minimum of $\Lb$ is $1$, we know that  
the Gram matrix of $B$ is contained in the affine subspace 
\begin{equation} \label{eqn:TC}
T_C
= 
\{
G \in \sn \; | \; G[\bar{e}^{(i)}] = 1 \mbox{ for } i=1,\ldots,n
\}
\end{equation}
of $\sn$, respectively in its intersection with the Ryshkov polyhedron $\cR$.
This intersection is a face $\cF$ of $\cR$ that is determined by $C$,
up to the choice of the basis $B$. Choosing another basis
$B'$, we find a matrix $U\in\GL_n(\Z)$ with $B'=BU$
and a corresponding face $\cF'$ of $\cR$ with $\cF'= U^t \cF U$.
Thus up to the action of $\GL_n(\Z)$, 
the face $\cF$ is uniquely determined by the code $C$.
The orbit of the relative interior of $\cF$ corresponds
to a uniquely determined minimal class $\cC_C$.
It has the desired property, as every pair of lattices $(\Lb,\Lb')$
satisfying the assumption of the proposition has a basis with
Gram matrix in $\cF$.
\end{proof}

Let us note that the face $\cF$ of the Ryshkov polyhedron 
described in the proof is bounded. In fact, it can be shown that
the bounded faces of the Ryshkov polyhedron are precisely the
ones coming from lattices having $n$ linearly independent minimum vectors 
(attaining the minimum~$1$).
So the classification of possible codes is equivalent to the 
classification of bounded faces of $\cR$ up to the action of
$\GL_n(\Z)$. 
For this it is enough to determine
bounded faces of maximal dimension, that is, those bounded faces 
that are themselves not contained in the boundary of other bounded faces.

\smallskip

An important tool that we use, to show that certain codes can
not be realized, is the estimation of the Hermite constant on a 
given minimal class.
The minimum of the Hermite constant may not be attained on a 
given minimal class, but if it is attained, then it is attained
at a {\em weakly eutactic} lattice. These lattices are 
characterized by the fact that a corresponding Gram matrix~$G$ 
satisfies
\begin{equation}  \label{eqn:eutaxy}
G^{-1} = \sum_{x\in S(G)} \lb_x x x^t. 
\end{equation}
for real coefficients $\lb_x$. 
A lattice is called 
{\em eutactic}, if there exists such a relation with strictly 
positive coefficients~$\lb_x$ and 
{\em strongly eutactic} if they are additionally all equal.
The above mentioned result is due to Anne-Marie Berg\'e and the second author 
(see \cite[Section~9.4]{martinet-2003}). They also show that there exists at most one weakly 
eutactic lattice in a given minimal class~$\cC$, respectively in its 
{\em closure}
$$
\ov{\cC}=\bigcup_{\cC\prec\cC'}\,\cC'\,
.
$$
An easy consequence is the following result for orthogonal sums 
of weakly eutactic lattices.

\begin{prop}\label{prop:sumorth} 
Let $\cC_1,\cC_2$ be minimal classes of dimensions $n_1,n_2$ 
and define $\cC:=\cC_1\bgo\cC_2$, of dimension $n=n_1+n_2$, by 
$$\cC=\{\Lb=\Lb_1\bgo\Lb_2\mid \Lb_i\in\cC_i,\, 
S(\Lb)=S(\Lb_1)\cup S(\Lb_2)\}\,.$$ 
Then the weakly eutactic lattices in $\ov{\cC_1}\bgo\ov{\cC_2}$ 
are the orthogonal sums $\Lb_1\perp\Lb_2$ of weakly eutactic lattices 
$\Lb_i\in\cC_i$. In particular, the minimum of $\g$ on $\cC$ 
is attained on an orthogonal sum $\Lb_1\perp\Lb_2$. 
\end{prop}

The following lemma derived from the proposition 
will allow us to show that certain codes are 
impossible for~$n=9$.

\begin{lemma}\label{lemsecE8} 
Let $\Lb$ be a well-rounded lattice of dimension $n=9$ having 
an $\E_8$-section with the same minimum. Then no lattice having 
the same minimum as $\Lb$ strictly contains~$\Lb$. 
\end{lemma} 

\begin{proof} 
Assume that some lattice $L$ with $\min L=\min\Lb$ contains $\Lb$ 
to an index $d\ge 2$. Let us scale for convenience all lattices 
to minimum~$2$. By Proposition~\ref{prop:sumorth}, we have 
$\g(\Lb)\ge\g(\E_8\perp\A_1)=2\cdot 2^{-1/9}$, hence 
\linebreak 
$\g(L)^{9/2}\ge 2 \cdot \g(\E_8\perp\A_1)^{9/2}=32$, 
which contradicts the upper bound~\eqref{eqn:boundongamma9and10}. 
\end{proof}

\section{An algorithm to check the existence of a code}  \label{sec:algorithm} 

The basic idea of the following algorithm is motivated by
the geometric situation described in the proof of Proposition~\ref{prop:smallclass}.
Given a $\Z/d\Z$-code $C$, we either show that the 
intersection of $T_C$ (as in~\eqref{eqn:TC}) with $\cR$ is empty,
or we show that it is non-empty by finding a corresponding Gram matrix.
A problem we have to deal with is the fact that $\cR$ is given by infinitely many inequalities.
The idea is to start with a finite set of inequalities and
then successively add inequalities until either we find a point in 
the intersection or have a proof for infeasibility.
For the starting set of inequalities we 
take a finite set of vectors $V\subset \Z^n$
such that the linear function $\Tr G = \langle \Id_n, G \rangle $ is bounded from above
on the polyhedron 
\begin{equation}  \label{eqn:feasregion}
P = \{ G \in T_C \; | \; G[v]\geq 1 \mbox{ for all $v\in V$ } \}
.
\end{equation}
Note, if $P$ is empty, we have a proof that the minimal class $\cC_C$ is empty.
The assumption on the bounded trace
allows us to find a solution of the {\em linear programming problem} 
\begin{equation}  \label{eqn:LP}
\max_{Q\in P} \Tr Q
.
\end{equation}
Depending on whether or not the found solution $G$ of this linear program 
is positive definite or not we have a different strategy for obtaining 
additional inequalities, respectively vectors for the description of $P$.
In the first case we compute the minimum of $G$. If it is $1$, the Gram
matrix~$G$ proves the existence of the minimal class $\cC_C$ and corresponding lattices.
If the minimum is less than $1$ we add $S(G)$ to $V$.
In the second case, if $G$ is not positive definite, we add some
vector(s) $v\in \Z^n$ to $V$ with $G[v]\leq 0$.
Such vectors can be found for example by an eigenvector computation.
Having enlarged $V$, we can go back and solve the linear program \eqref{eqn:LP},
now with respect to a smaller polyhedron~$P$. Again, if~$P$ is non-empty,
we obtain a new solution $G$ and proceed as described above.
See Algorithm~\ref{alg:index-algorithm} for a schematic description
of the described procedure. Note that all of the steps can be realized 
with the help of a Computer Algebra System. 
We used {\tt MAGMA}~\cite{magma} for our computations, 
in conjunction with {\tt lrs}~\cite{lrs} to perform polyhedral computations.

\begin{bigalg}   \label{alg:index-algorithm}
\fbox{
\begin{minipage}{12.0cm}
\begin{flushleft}
\smallskip
\textbf{Input:} $n$, $d$, $C = \{ a^{(i)} \in (\Z/d\Z)^n \; | \; i=1,\ldots,k \}$\\ 
\textbf{Output:} (``true'' and a corresponding Gram matrix $G$)\\ 
\hspace*{1.2cm} or (``false'' and $V\subset\Z^n$ such that $P$ in \eqref{eqn:feasregion} is empty)\\                  
\bigskip

$V =$ initial set of integral, non-zero vectors\\
{\bf do}\\
\hspace*{0.5cm} $P = \{G \in T_C : G[v]\geq 1 \mbox{ for all $v\in V$ } \}$\\
\hspace*{0.5cm} {\bf if} $P=\emptyset$\\ 
\hspace*{0.5cm} {\bf then}\\ 
\hspace*{1.0cm} {\bf return} (``false'', $V$)\\
\hspace*{0.5cm} {\bf else}\\
\hspace*{1.0cm} determine $G\in P$ with $\displaystyle \Tr G = \max_{G'\in P} \Tr G'$\\
\hspace*{1.0cm} {\bf if} $G\in\sno$\\ 
\hspace*{1.0cm} {\bf then}\\ 
\hspace*{1.5cm} {\bf if} $\min (G) \geq 1$\\
\hspace*{1.5cm} {\bf then}\\  
\hspace*{2.0cm} {\bf return} (``true'', $G$)\\
\hspace*{1.5cm} {\bf else}\\ 
\hspace*{2.0cm} $V=V\cup S(G)$\\
\hspace*{1.5cm} {\bf end if}\\ 
\hspace*{1.0cm} {\bf else}\\  
\hspace*{1.5cm}  compute finite $\mbox{NV}\subset \Z^n$ with $G[v]\leq 0$ for $v\in \mbox{NV}$\\
\hspace*{1.5cm}  $V=V\cup \mbox{NV}$\\ 
\hspace*{1.0cm} {\bf end if}\\ 
\hspace*{0.5cm} {\bf end if}\\  
{\bf end do}

\end{flushleft}
\end{minipage}
}
\\[1ex]
{\bf Algorithm \arabic{alg}.} Determines feasibility of a given code C
\end{bigalg}

It is not {\em a priori} clear that this procedure is in fact an algorithm, that is, 
if it stops after finitely many steps. As long as it does in all cases we consider,
we may not even care. In order to guarantee that the computation finishes after finitely 
many steps, depending on $T$, we can restrict the vectors to be added to $V$
to some large finite subset of $\Z^n$, for example by bounding the absolute 
value of coordinates. An explicit bound for the coordinates
of vectors $x\in \Z^n$, with $G[x] = 1$ for $G$ in $\cR$ with $\Tr G$ bounded by some constant, 
is derived in~\cite[Section~3.1]{schuermann-2008}.

Algorithm~\ref{alg:index-algorithm} yields a vertex $G$ of 
the face $\cF_C = T_C \cap \cR$ of $\cR$, associated with the code $C$
through the choice of specific coordinates $\bar{e}^{(i)}$
(see the proof of Proposition~\ref{prop:smallclass}).
If we want to know a description of the whole face $\cF_C$,
we can compute all of its vertices by 
exploring neighboring vertices of vertices found so far,
until no new vertices are discovered. As the face $\cF_C$ is bounded 
this traversal search on the graph 
of vertices and edges (one-dimensional faces) of $\cF_C$ ends after 
finitely many steps.
Given a vertex $G$, the neighboring vertices are found as follows.
We consider the polyhedral cone 
\begin{equation}  \label{eqn:Gcone}
\{
G' \in T_C 
\; | \;
G'[x] \geq 1 \mbox{ for all } x\in S(G)
\}
\end{equation}
with apex $G$. Thus the elements of $S(G)$ yield a polyhedral description 
with linear inequalities.
Using standard methods (cf. for example \cite[Appendix A]{schuermann-2008}),
we can convert it into a description
$$
\{
G' \in \sn
\; | \;
G' = G + \lambda_1 R_1 + \ldots + \lambda_k R_k, \lambda_i \geq 0
\}
$$
with {\em extreme rays} given by generators $R_1,\ldots, R_k\in \sn\setminus\{0\}$.
For each of these generators $R_i$ we can find a neighboring vertex 
of $G$ in $\cF_C$ by a procedure similar to the one of finding
contiguous perfect forms (cf. \cite[Section~3.1]{schuermann-2008}).
See Algorithm~\ref{alg:determination-of-contiguous-vertex}.

\begin{bigalg}   \label{alg:determination-of-contiguous-vertex}
\fbox{
\begin{minipage}{12cm}
\begin{flushleft}
\smallskip
\textbf{Input:} Vertex $G$ of $\cF_C$ and generator $R$ of an extreme ray of \eqref{eqn:Gcone}\\
\textbf{Output:} $\rho>0$ with $\min(G + \rho R) = \min(G)$ and $S(G + \rho R) \not \subseteq S(G)$.\\
\medskip
$(l, u) \leftarrow (0,1)$\\
\smallskip
\textbf{while} $G + u R \not\in \sno$ or $\min(G + u R) = \min(G)$ \textbf{do}\\
\hspace{2ex} \textbf{if} $G + u R \not\in \sno$ \textbf{then} $u \leftarrow (l + u)/2$\\
\hspace{2ex} \textbf{else} $(l,u) \leftarrow (u, 2u)$\\
\hspace{2ex} \textbf{end if}\\
\textbf{end while}\\
\smallskip
\textbf{while} $S(G + l R) \subseteq S(G)$ \textbf{do}\\
\hspace{2ex} $\gamma \leftarrow \frac{l + u}{2}$\\
\hspace{2ex} \textbf{if} $\min (G + \gamma R) \geq \min(G)$ \textbf{then} $l \leftarrow \gamma$\\
\hspace{2ex} \textbf{else} \\
\hspace{4ex} $u \leftarrow \min \left\{ \displaystyle \frac{\min(G)-G[x]}{R[x]} \; | \;
                             x\in S (G+\gamma R), R[x] < 0 \right\} \cup \{ \gamma \} $\\
\hspace{2ex} \textbf{end if}\\
\hspace{2ex} \textbf{if} $\min (G + \gamma R) = \min(G)$ \textbf{then} $l \leftarrow u$ \textbf{end if}\\
\textbf{end while}\\
\smallskip
$\rho \leftarrow l$
\end{flushleft}
\end{minipage} 
}
\\[1ex]
{\bf Algorithm \arabic{alg}.} Determination of neighboring vertices of $\cF_C$.
\end{bigalg}

Once we know all the vertices of $\cF_C$, we can easily compute a relative interior point
that carries information on several invariants of the class $\cC_C$, 
like its perfection rank $r$ and the number $s$ of minimal vectors of $\Lb$.
To obtain just any interior point, it is actually enough to know
an initial vertex and the generating rays of the polyhedral cone \eqref{eqn:Gcone}.
If we know all the vertices $G_1,\ldots, G_k$ of $\cF_C$, 
we can compute the vertex barycenter $\frac{1}{k}\sum_{i=1}^k G_i$ of $\cF_C$
that carries even more information. For example its automorphism group is equal
to the automorphism group 
$$
\Aut \cF_C
=
\{
U\in\GL_n(\Z)
\; | \;
U^t \cF_C U = \cF_C
\}
$$ 
of $\cF_C$. This is due to the fact that any automorphism of the face $\cF_C$ permutes its
vertices, and hence leaves the vertex barycenter fixed. On the other hand, 
the vertex barycenter is a relative interior point of~$\cF_C$, that is a face of the
Ryshkov polyhedron on which any element of $\GL_n(\Z)$ acts. Therefore, 
for topological reasons, any automorphism of the vertex barycenter has 
to be an automorphism of $\cF_C$.

In higher dimensions, i.e. for $n=9$, 
depending on the face $\cF_C$,
the polyhedral computations necessary to find an initial vertex~$G$ or even
all vertices may not be feasible (within a reasonable amount of time).
In these cases, we can try to exploit available symmetries,
that is, use the group $\Aut \cF_C$. 
It can be computed from the 
coordinate vectors $\bar{e}^{(i)}$ 
which define the linear space $T_C$ (see~\eqref{eqn:TC}).
In fact, 
$$
\Aut \cF_C =
\{
U\in\GL_n(\Z)
\; | \;
U\bar{e}^{i} \in \{ \bar{e}^{(1)},\ldots, \bar{e}^{(n)} \} 
\mbox{ for all } i = 1,\ldots, n
\}
.
$$
As at least the vertex barycenter of $\cF_C$ is 
invariant with respect to~$\cG=\Aut \cF_C$,
it is contained in the {\em $\cG$-invariant linear subspace} 
$$
T_{\cG} = \{
G\in \sn \; | \;
U^t G U = G \mbox{ for all } U \in \GL_n(\Z)
\}
.
$$
So if we just want to check the feasibility of a given code $C$
and want to compute its invariants from the vertex barycenter, 
then we can restrict the search to the linear space $T_{\cG}$, 
respectively to the affine space $T_{\cG}\cap T_C$.
In practice, in many cases the computation time is reduced 
tremendously by this kind of symmetry reduction.

Note that $T_{\cG} \cap \cR$, like $\cR$ itself, is a locally finite polyhedron.
Its vertices (and corresponding lattices) are called {\em $\cG$-perfect}.
We refer to \cite{schuermann-2009} for a detailed account and interesting examples.
If there is only one Gram matrix up to scaling 
in $T_{\cG}\cap T_C\cap \cR$ it is {\em $\cG$-eutactic}
and therefore {\em eutactic} (see~\cite{martinet-2003} for details).
By the discussion at the end of Section~\ref{sec:quadforms},
we can then conclude that the minimum of $\gamma$ for the minimal class $\cC_C$
is attained on it.

\section{Restricting the number of codes under consideration}\label{sec:restrictnumberofcodes} 

\smallskip 

The computations proposed in the last sections are quite involved,
so it is desirable to {\em a priori} exclude as many cases as possible.
An efficient basic tool to restrict the number of possible codes  
is the following identity.

\begin{prop}{\rm(Watson, \cite{watson-1971}.)}\label{propwatid} 
Let $e_1,\dots,e_n$ be independent vectors in~$E$, 
let $a_1,\dots,a_n$ and $d\ge 2$ be integers, 
and let 
$$f=\frac{a_1 e_1+\dots+a_n e_n}d\,.$$ 
Denote by $\sgn(x)$ the sign of the real number~$x$. Then, 
$$\left( \left(\sum_{i=1}^n\abs{a_i} \right)-2d\right)N(f)=
\sum_{i=1}^n\abs{a_i}\left(N \left(f-\sgn(a_i)e_i \right)-N(e_i)\right)\,.$$
\end{prop} 

\begin{proof} 
Just develop both sides of the displayed formula. 
\end{proof}

\begin{corol}\label{corwatid}{\rm(Watson, \cite{watson-1971}.)} 
With the notation above, assume that the $e_i$ are minimal vectors 
of a lattice $\Lb$, that $f$ belongs to $\Lb$ and that the $a_i$ 
are non-zero. Then we have 
$$\sum_i\,\abs{a_i}\ge 2d\,,$$ 
and equality holds \ifff the $n$~vectors $e'_i=f-e_i$ are minimal. 
\end{corol}

When adding vectors $f$ as above to the lattice 
$\Lb'=\la e_1,\dots,e_n\ra$, one may always reduce the $a_i$ 
modulo~$d$. When there is only one such vector, i.e. when we may write 
$\Lb=\la\Lb',f\ra$, then by negating some $e_i$ if needed, we may 
moreover assume that all $a_i$ are non-negative. By reducing 
the dimension, we even may assume they are strictly positive. In this case, 
we adopt the following notation:

\begin{notation}\label{notacyclic}{\rm 
Suppose that $\Lb/\Lb'$ is cyclic of order $d\ge 2$, and that 
$$\Lb=\la\Lb',f\ra\ \text{ with }\ 
f=\frac{a_1 e_1+\dots+a_n e_n}d$$ 
and $a_i\in\{\pm 1,\dots,\pm\lf\frac d2\rf\}$. 
For $i=1,\ldots,\lf\frac d2\rf$ we then set 
$$
m_i=\abs{\{a_j\mid a_j=\pm i\}}
$$ 
and say that $\Lb$ {\em is of type} $(m_1,\dots,m_{\lfloor d/2\rfloor})_d$, 
or simply $(m_1,\dots,m_{\lfloor d/2\rfloor})$. 
}
\end{notation}

Note, when we use this notation, we have $m_1+\dots+m_{\lfloor d/2\rfloor }=n$.  
It will be generally assumed that $d$ and the $a_i$ are coprime,
because otherwise, we could replace $d$ by one of its strict divisors.

It should be noted that we also have $\Lb=\la\Lb',a\,f\ra$ 
for any $a$ coprime to~$d$. This induces an action of 
$(\Z/d\Z)^\times/\{\pm 1\}$ on the set of admissible types 
$(m_1,\dots,m_{\lfloor d/2\rfloor})_d$. If one of the types in an orbit
does not satisfy Watson's criterion in Corollary~\ref{corwatid},
we know that a corresponding code does not exist.

%
%
%

\section{Classifying cyclic quotients}\label{sec:cyclic}

In this section we give complete results on cyclic quotients for dimension~$9$. 
The results are displayed in Table~\ref{tab:cyclic},
with coordinates $(a_1,\ldots,a_9)$ of a generator,
together with three basic invariants of lattices in the corresponding 
minimal class (see Proposition~\ref{prop:smallclass}):
$s=s(\Lb)$, $r=\perf(\Lb)$ and $s'=s(\Lb')$.
Note that we list only one admissible type $(m_1,\dots,m_{\lfloor d/2\rfloor})$ 
of each $(\Z/d\Z)^\times/\{\pm 1\}$ orbit,
as explained at the end of the previous section.

Our results show that cyclic quotients exist for $n=9$ only with $d\leq 10$ and $d=12$.
They were obtained using an implementation of Algorithm~\ref{alg:index-algorithm},
using {\tt MAGMA} scripts in conjunction with {\tt lrs}.
Our source code can be obtained from the online appendix of this paper,
contained in the source files of its arXiv version {\tt arXiv:0904.3110}.
We used a {\tt C++} program to systematically generate a list of possible cases
satisfying the conditions of Watson, described in Section~\ref{sec:restrictnumberofcodes}.
We checked all cases $d\leq 30$, left by Proposition~\ref{propmaxind} 
and the known bound~\eqref{eqn:boundongamma9and10} on $\gamma_9$.
In this way, we also confirmed all of the previously known results for $n\leq 8$ in~\cite{martinet-2001}.
In dimension~$9$, we found several new possible indices. Below we exemplary 
give a detailed account of our computational result for $d = 12$.
For $d\leq 6$ we give a derivation.

\subsection{Cyclic cases with $\mathbf{d = 12}$}

\label{sec:cyclicd12}

Among the most interesting
cases are the cyclic quotients with~$d=12$. 
There are four different types listed in Table~\ref{tab:cyclic}.
Three of them occur only for the laminated lattice~$\Lb_9$;
it is the unique lattice in dimension~$9$ with~$s=136$.
The fourth entry shows that there is also one type
that occurs for some lattices with~$s\geq 87$.
One of them with~$s=87$ is the lattice~$L_{87}$  
with Gram matrix
\begin{small}
\begin{equation*}   
\begin{pmatrix}
   4  & 2  &  -2  &  2  &  -2  &   0  & 2 & -2  &  2  \\
 2  &  4  & -1  &  0  & -2  &  2  &   0  &  0  & 2 \\
-2 & -1  &   4  & -2 &  2  & -1  & -1 &  1 & -2 \\
 2  &   0  & -2  &   4  &    0  &    0 &  2 & -1 &  2 \\
-2 & -2 &  2 &   0 &   4 & -2 &   0 &  1 & -2 \\
   0  & 2 & -1 &   0 & -2  &  4  &  0  & 1 & 2 \\
 2  &  0 & -1 &  2 &   0  &   0  &   4  & -2  &   0 \\
-2  &  0 & 1  & -1 &  1  & 1 & -2  &  4  &  0 \\
 2  & 2 & -2 & 2  & -2 &  2 &   0 &   0  &   4 \\
\end{pmatrix}
.
\end{equation*}
\end{small}

The perfection rank of $L_{87}$ and its Gram matrix~$G$ is~$42$. 
Hence, by the discussion in Section~\ref{sec:quadforms}, 
it is in the relative interior of a three-dimensional
face of the Ryshkov polyhedron. A closer analysis reveals that this face is an 
octahedron with centroid~$G$. Its six vertices come in opposite
pairs. Two of these pairs contain Gram matrices of~$\Lb_9$.
They are obtained as $G \pm R$ and $G\pm R'$,
with $R$ and $R'$ being symmetric and having entries~$0$ everywhere, 
except at the positions determined by the conditions 
$R_{32} = R_{36}=R_{37}=1$ and $R'_{84} = R'_{85}=-R'_{86}=1$ (together with the symmetric ones).
The other pair $G \pm R''$, with $R''_{38}=R''_{83}=1$ and $0$ elsewhere,
contains two Gram matrices of another perfect
lattice with $s=99$. We call this special perfect lattice~$L_{99}$ in the sequel. 
It is characterized by the fact that it is the only perfect lattice aside of~$\Lb_9$
that has a Minkowskian sublattice with cyclic quotient of order~$d=12$.
Note that any Gram matrix of a lattice with cyclic quotient of order~$d=12$ 
can be obtained as a convex combination of suitable Gram matrices
of~$L_{99}$ and~$\Lb_9$.

A computer assisted calculation shows that both lattices~$L_{87}$ and~$L_{99}$
are eutactic but not strongly eutactic (see~\eqref{eqn:eutaxy}).
For example, for a Gram matrix $G$ of~$L_{99}$ we compute (using {\tt MAGMA}) the set of minimal vectors
$S(G)\subset \Z^n$ and find that it falls into five orbits under the action of the
automorphism group of~$G$. Each orbit $O$ yields a barycenter $\sum_{x\in O} x x^t$
and the so obtained five barycenters $b_1,\ldots,b_5$ satisfy a relation
$G^{-1} = \lb_1 b_1 + \ldots + \lb_5 b_5$ for positive coefficients $\lb_i$,
as can be checked easily for example with the Maple package {\tt Convex}~\cite{convex}.
See the comments in the LaTeX source file of the arXiv paper  {\tt arXiv:0904.3110} for further details.
For the perfect lattice~$L_{99}$, its eutaxy implies (by a theorem of Voronoi; see~\cite{martinet-2003})
that it is extreme, that is, it attains a local maximum of the Hermite invariant.

\begin{table}

\noindent
\begin{minipage}{6cm}

\begin{tabular}{|c|c|c|c|c|}
\hline
$d$  &  generator  & $s$  &  $r$  &  $s'$  \\
\hline
\hline
 2  & (1,1,1,1,1,1,1,1,1)  &  9  &  9  &  9 \\
 3  & (1,1,1,1,1,1,1,1,1)  &  9  &  9  &  9 \\
 4  & (1,1,1,1,1,1,1,1,1)  &  9  &  9  &  9 \\
 4  & (1,1,1,1,1,1,1,1,2)  &  9  &  9  &  9 \\
 4  & (1,1,1,1,1,1,1,2,2)  &  9  &  9  &  9 \\
 4  & (1,1,1,1,1,1,2,2,2)  &  9  &  9  &  9 \\
 4  & (1,1,1,1,1,2,2,2,2)  &  9  &  9  &  9 \\
 4  & (1,1,1,1,2,2,2,2,2)  &  17  &  15  &  9 \\
 5  & (1,1,1,1,1,1,1,1,2)  &  18  &  17  &  9 \\
 5  & (1,1,1,1,1,1,1,2,2)  &  9  &  9  &  9 \\
 5  & (1,1,1,1,1,1,2,2,2)  &  9  &  9  &  9 \\
 5  & (1,1,1,1,1,2,2,2,2)  &  9  &  9  &  9 \\
 6  & (1,1,1,1,1,1,2,2,2)  &  18  &  17  &  9 \\
 6  & (1,1,1,1,1,2,2,2,2)  &  9  &  9  &  9 \\
 6  & (1,1,1,1,2,2,2,2,2)  &  23  &  20  &  9 \\
 6  & (1,1,1,1,1,1,1,2,3)  &  27  &  25  &  9 \\
 6  & (1,1,1,1,1,1,2,2,3)  &  9  &  9  &  9 \\
 6  & (1,1,1,1,1,2,2,2,3)  &  9  &  9  &  9 \\
 6  & (1,1,1,1,2,2,2,2,3)  &  9  &  9  &  9 \\
 6  & (1,1,1,2,2,2,2,2,3)  &  17  &  15  &  9 \\
 6  & (1,1,1,1,1,1,1,3,3)  &  9  &  9  &  9 \\
 6  & (1,1,1,1,1,1,2,3,3)  &  9  &  9  &  9 \\
 6  & (1,1,1,1,1,2,2,3,3)  &  9  &  9  &  9 \\
 6  & (1,1,1,1,2,2,2,3,3)  &  9  &  9  &  9 \\
 6  & (1,1,1,2,2,2,2,3,3)  &  9  &  9  &  9 \\
 6  & (1,1,2,2,2,2,2,3,3)  &  17  &  15  &  9 \\
 6  & (1,1,1,1,1,1,3,3,3)  &  15  &  14  &  9 \\
 6  & (1,1,1,1,1,2,3,3,3)  &  23  &  20  &  9 \\
 6  & (1,1,1,1,2,2,3,3,3)  &  15  &  14  &  9 \\
 6  & (1,1,1,2,2,2,3,3,3)  &  15  &  14  &  9 \\
 6  & (1,1,2,2,2,2,3,3,3)  &  15  &  14  &  9 \\
 6  & (1,2,2,2,2,2,3,3,3)  &  23  &  20  &  9 \\
7  &  (1,1,1,1,1,2,2,2,3)  &  33  &  31  &  9 \\  
7  &  (1,1,1,1,2,2,2,2,3)  &  18  &  17  &  9 \\ 
7  &  (1,1,1,1,1,1,2,3,3)  &  136  &  45  &  24 \\ 
7  &  (1,1,1,1,1,2,2,3,3)  &  9  &  9  &  9 \\ 
7  &  (1,1,1,1,2,2,2,3,3)  &  9  &  9  &  9 \\ 
7  &  (1,1,1,1,1,2,3,3,3)  &  30  &  26  &  9 \\ 
7  &  (1,1,1,1,2,2,3,3,3)  &  9  &  9  &  9 \\ 
\hline
\end{tabular}

\end{minipage}
\hfill
\begin{minipage}{6cm}

\begin{tabular}{|c|c|c|c|c|}
\hline
$d$  &  generator  & $s$  &  $r$  &  $s'$  \\
\hline
\hline
7  &  (1,1,1,2,2,2,3,3,3)  &  18  &  17  &  9 \\ 
8  &  (1,1,1,1,2,2,2,3,3)  &  136  &  45  &  18 \\ 
8  &  (1,1,1,2,2,2,2,3,3)  &  9  &  9  &  9 \\ 
8  &  (1,1,2,2,2,2,2,3,3)  &  35  &  28  &  9 \\ 
8  &  (1,1,1,2,2,2,3,3,3)  &  50  &  37  &  12 \\ 
8  &  (1,1,1,1,1,2,2,3,4)  &  136  &  45  &  19 \\ 
8  &  (1,1,1,1,2,2,2,3,4)  &  40  &  34  &  9 \\ 
8  &  (1,1,1,2,2,2,2,3,4)  &  37  &  32  &  9 \\ 
8  &  (1,1,1,1,2,2,3,3,4)  &  25  &  24  &  9 \\ 
8  &  (1,1,1,2,2,2,3,3,4)  &  9  &  9  &  9 \\ 
8  &  (1,1,2,2,2,2,3,3,4)  &  17  &  15  &  9 \\ 
8  &  (1,1,1,1,2,3,3,3,4)  &  31  &  29  &  9 \\ 
8  &  (1,1,1,2,2,3,3,3,4)  &  27  &  25  &  9 \\ 
8  &  (1,1,1,1,3,3,3,3,4)  &  34  &  30  &  9 \\ 
8  &  (1,1,1,1,2,2,3,4,4)  &  32  &  28  &  9 \\ 
8  &  (1,1,1,2,2,2,3,4,4)  &  38  &  29  &  9 \\ 
8  &  (1,1,1,1,2,3,3,4,4)  &  42  &  34  &  10 \\ 
8  &  (1,1,1,2,2,3,3,4,4)  &  9  &  9  &  9 \\ 
8  &  (1,1,2,2,2,3,3,4,4)  &  33  &  27  &  9 \\ 
8  &  (1,1,1,2,3,3,3,4,4)  &  23  &  21  &  9 \\ 
 9  &  (1,1,1,2,2,2,3,3,4)  &  84  &  43  &  13 \\  
 9  &  (1,1,1,2,2,3,3,3,4)  &  50  &  37  &  10 \\  
 9  &  (1,1,1,1,2,3,3,4,4)  &  136  &  45  &  16 \\  
 9  &  (1,1,1,2,2,3,3,4,4)  &  53  &  37  &  10 \\  
 9  &  (1,1,1,2,3,3,3,4,4)  &  31  &  27  &  9 \\  
 9  &  (1,1,2,2,3,3,3,4,4)  &  15  &  14  &  9 \\  
 10  & (1,1,2,2,2,2,3,3,5)  &  136  &  45  &  16 \\  
 10  & (1,1,2,2,2,2,3,4,5)  &  136  &  45  &  16 \\  
 10  & (1,1,2,2,2,3,3,4,5)  &  64  &  40  &  10 \\  
 10  & (1,1,1,2,2,3,4,4,5)  &  136  &  45  &  13 \\  
 10  & (1,1,2,2,2,3,4,4,5)  &  51  &  36  &  9 \\  
 10  & (1,1,2,2,3,3,4,4,5)  &  43  &  39  &  9 \\  
 10  & (1,1,1,2,3,4,4,4,5)  &  84  &  43  &  12 \\  
 10  & (1,1,2,2,3,4,4,4,5)  &  53  &  37  &  9 \\  
 12  & (1,1,2,3,3,4,4,5,6)  &  136  &  45  &  12 \\ 
 12  & (1,1,3,3,4,4,5,5,6)  &  136  &  45  &  12 \\  
 12  & (1,2,2,3,3,3,4,4,5)  &  136  &  45  &  13 \\ 
 12  & (1,2,2,3,3,4,4,5,6)  &  87  &  42  &  10 \\ 
       & & & & \\
\hline
\end{tabular}

\end{minipage}

\medskip
\caption{Cyclic cases for $n=9$}
\label{tab:cyclic}
\end{table}

\subsection{Cyclic cases with $\mathbf{d\leq 6}$}

We give arguments for the $9$-dimensional cases with $d\leq 6$ below.
Actually, we shall see that it is possible to give complete results 
for all dimensions with little extra work once we know the results up 
to $n=8$ for $\imath\le 5$, and up to $n=9$ for $\imath = 6$. 
The general strategy to deal with cyclic quotients of order 
$\imath=d=3$ to~$6$ is as follows 
(the notation is that of \ref{notacyclic}): taking into account 
Corollary~\ref{corwatid} and the action of $(\Z/5\Z)^\times$ 
which allows us to exchange $m_1$ and $m_2$ when $d=5$, 
we obtain the inequalities 
$n=m_1\ge 6$ for $d=3$, $m_1\ge 4$ and $m_1+2m_2\ge 8$ for $d=4$, 
$m_1\ge \frac n2$ and $m_1+2m_2\ge 10$ for $d=5$ 
(and $n\ge 8$ is known; see \cite{martinet-2001}), and finally $m_1+m_2\ge 4$, $m_1+m_3\ge 6$ 
and $m_1+2m_2+3m_3\ge 12$ for $d=6$.

As a next step we apply the following averaging argument, justified by the 
discussion on symmetry at the end of Section~\ref{sec:algorithm}:

\begin{remark} \label{rem:averaging}
Assuming a lattice of type $(m_1,\dots,m_{d/2})_d$ exists, 
with the notation of \ref{notacyclic}, 
we may assume that the scalar products
of minimum vectors $e_k$ and $e_l$ associated with $m_i$ and $m_j$ 
are equal. For $i=j$, that is, for $e_k$ and $e_l$ associated with the
same $m_i$, we denote this scalar product by $x_i$;
for $i\not = j$, we denote it by $y_{i,j}$.
These parameters are omitted if $m_i=0,1$, respectively $m_j=0,1$,
except of $y_{i,j}$ in the case $m_i=m_j=1$.
For even~$d$, we can additionally set $x_{d/2}=y_{d/2,j}=0$, 
as we can average the two Gram matrices
resulting from a base change, which replaces only $e_k$ by $-e_k$.
\end{remark}

By this kind of averaging argument, we assume that Gram matrices 
for a lattice type depend only on a short list of parameters.
In particular, only on 
$x_1$ if $d=3$ or $d=4$ (and then $x_1=0$ if $m_1=4$), 
at most three parameters $x_1,x_2,y_{1,2}$ if $d=5$ or $d=6$. 
In practice, for large enough $m_i$, the existence and the equalities 
$s=r=n$ hold taking pairwise orthogonal vectors~$e_i$, 
so that a finite number of verifications will suffice, 
which need difficult arguments only in low dimensions. 

\smallskip 

\noi\underbar{Index $2$.}
For $\imath=2$, there is one lattice, namely $\Lb=\la\Lb',f\ra$ 
where $f=\frac{e_1+\dots+e_n}2$, which can be constructed 
for $n\ge 4$ using an orthogonal basis for $\Lb'$.  
We have $(s,r)=(12,10)$ and $\Lb\sim\D_4$ if $n=4$. 
For $n\ge 5$ we get $s=r=n$. 

\smallskip 

\noi\underbar{Index $3$.} 
Here, we have $\Lb=\la\Lb',f\ra$ and $f=\frac{e_1+\dots+e_n}3$. 
By Corollary~\ref{corwatid}, we must have $n\ge 6$ and $s\ge 12$ 
when $n=6$ (because the vectors $e_i$ and $e'_i=f-e_i$ are minimal; 
this also shows that the index system 
of $S_6=\{e_i,e'_i : i=1,\ldots,6 \}$ is $\cI(S_6)=\{1,2,3\}$), 
and~\cite[Proposition~3.5]{bm-2009} shows that $\perf(S_6)=11$. 
The result is: 
{\em
Lattices exist \ifff $n\ge 6$, and we have $(s,r)=(12,11)$ 
if $n=6$, $s=r=n$ if $n\ge 7$.
} 
For the proof, it suffices to find convenient values 
for the common value $x_1$ of scalar products $e_i\cdot e_j$. 
We may clearly choose $x_1=0$ for $n\ge 10$, 
and we check that for $x_1=\frac 15$, we have $s=12$ if $n=6$ and $s=n$
if $n=7,8,9$. 

\smallskip 

\noi\underbar{Index $4$.} 
Here we have $\Lb=\la\Lb',f\ra$ and 
$f=\frac{e_1+\dots+e_{m_1}+2(e_{m_1+1}+\dots+e_n)}4$. 
The result is: 
{\em 
Lattices exist \ifff $m_1\ge 4$, $n\ge 7$ and $(m_1,m_2)\ne (7,0)$, 
and we have $(s,r)=(n+8,n+6)$ if $m_1=4$ and $s=r=n$ if $m_1\ge 5$, 
except in the following three cases: 
$(s,r)=(23,19)$ if $(m_1,m_2)=(4,3)$, 
$(s,r)=(21,19)$ if $(m_1,m_2)=(6,1)$, 
and $(s,r)=(16,15)$ if 
$(m_1,m_2)=(8,0)$.
} 
For the proof, it suffices to choose $x_1=0$ if $m_1=4$ 
or $m_1+4m_2\ge 17$ and $x_1=1/5$ otherwise. 

\smallskip 

\noi\underbar{Index $5$.} 
Recall that we assume that $m_1\ge\frac n2$. The result is: 
{\em Lattices exist \ifff $n=8$ and $m_1=4,5,6$, or 
$n=9$ and $m_1=5,6,7,8$, or $n\ge 10$ 
and we then have $s=r=n$ except in the four cases 
$(m_1,m_2)=(4,4)$, $(6,2)$, $(8,1)$ and $(10,0)$ where $(s,r)=(2n,2n-1)$.
} 
Watson's conditions together with $n\ge 8$ suffice 
to ensure the existence of lattices, and the special values 
for $(s,r)$ occur exactly when equality holds in Watson's 
Proposition~\ref{propwatid} and an analogue of Zahareva 
(see~\cite[proof of Proposition~9.1]{martinet-2001}).

\smallskip 

\noi\underbar{Index $6$.} 
Here the statement of the result is much more complicated, 
and for the sake of simplicity, we consider separately 
the case of dimension~$8$. 

\noi{\em Lattices exist \ifff $n=8$ and $(m_1,m_2,m_3)$ is one 
of the six sets listed in \cite[Table~11.1]{martinet-2001}; 
see the list below, or} 

\ctl{
{\em $n=9$, $m_1+m_2\ge 6$ and $m_1+m_3\ge 4$}.} 

\noi 
{\em When these conditions are satisfied, one has $(s,r)=(n+6,n+5)$ 
if $m_1+m_2=6$, $(s,r)=(n+8,n+6)$ if $m_1+m_3=4$, and $s=r=n$ 
otherwise, except in the following exceptions, for which we list 
$(m_1,m_2,m_3)$ and $(s,r)$:} 

\begin{itemize}
\item[$n=8$:] $(3,4,1)$: $(31,26)$; $(4,3,1)$: $(27,25)$; 
$(5,2,1)$: $(120,36)$ ($\Lb=\E_8$); $(2,4,2)$: $(28,22)$; 
$(4,2,2)$: $(36,28)$. 

\item[$n=9$:] $(4,5,0)$: $(23,20)$; $(6,3,0)$: $(18,17)$; 
$(7,1,1)$: $(27,25)$; $(1,5,3)$: 
\linebreak 
$(23,20)$; $(5,1,3)$: $(23,20)$. 

\item[$n=10$:] $(8,2,0)$: $(20,19)$; $(9,0,1)$: $(30,28)$. 

\item[$n=11$:] $(10,1,0)$ ,$( 22,21)$. 

\item[$n=12$:] $(12,0)$: $(24,23)$. 

\end{itemize}

%
%
%

\section{Classifying non-cyclic quotients}\label{sec:noncyclic}

In this section we give complete results on non-cyclic quotients for dimension~$9$.
For non-cyclic cases one only needs to consider cases with $d$ being a product of $k\geq 2$ 
numbers $d_1,\ldots, d_k$ that share a common divisor greater $1$. Otherwise we could
reduce the minimal number $k$ of necessary generators.
By Proposition~\ref{propmaxind} together with  
bound~\eqref{eqn:boundongamma9and10} on $\gamma_9$,
we only need to consider products $d\leq 30$.
We shall consider for all dimensions, 
quotients of type $2^k$, $3^k$, and $4\cdot 2^k$, 
which we have been able to compute by hand.
This leaves us with the following list of additional possible cases:
$6\cdot 2$, $8\cdot 2$, $10\cdot 2$, $12\cdot 2$, $14\cdot 2$, $4^2$, $6\cdot 3$,
$6\cdot 2^2$, $5^2$ and $9\cdot 3$.

The case $14\cdot 2$ can be excluded directly from 
our classification of cyclic quotients in Section~\ref{sec:cyclic},
as there exists no cyclic quotient of order~$14$. 
The cyclic quotients of order $12$ occur either on the similarity
class of $\Lb_9$ or on a minimal class containing that of the lattice
that we baptized~$L_{87}$ in Section~\ref{sec:cyclicd12}. 

This lattice is
weakly eutactic (an easily checked linear condition)
so that $\g(L_{87})$ is minimal among all lattices containing
its class, hence among all classes having a cyclic quotient
of order~$12$, since $\g(L_{87})$ is smaller than $\g(\Lb_9)$.
Doubling its density would produce a lattice in dimension~$9$ 
with Hermite invariant $2.16...$,
which contradicts the Cohn-Elkies bound~\eqref{eqn:boundongamma9and10}.

Using massive computer calculations to be explained below, 
we were able to exclude the cases 
$8\cdot 2$, $10\cdot 2$, $6\cdot 3$, $9\cdot 3$, $5 \cdot 5$,
and classifying all possible codes for $6\cdot 2$ and $4^2$.
From the classification of $6\cdot 2$, the last remaining 
case $6\cdot 2^2$ can be excluded, as we shall explain.

For the convenience of the reader, as in the cyclic case,
all of the existing cases are listed in tables
(see Tables~\ref{tab:n9d22-table}, \ref{tab:n9d222-table}, \ref{tab:n9d2222-table},
 \ref{tab:n9d33-table},
 \ref{tab:n9d42-table}, \ref{tab:n9d422-table}, 
 \ref{tab:n9d62-table}, and~\ref{tab:n9d44-table})
with coordinates for the generators,
together with the three basic invariants 
$s=s(\Lb)$, $r=\perf(\Lb)$ and $s'=s(\Lb')$
of lattices $\Lb$ in the corresponding minimal class.

\subsection{2-elementary quotients: preliminary remarks} 
\label{subsecperfelem} 

Consider a lattice $\Lb'$ of minimum~$1$ equipped with a basis 
$(e_1,\dots,e_n)$ of minimal vectors. 
Quotients $\Lb/\Lb'$ which are $2$-elementary are of the form 
$\Lb=\la\Lb',x/2 : x\in C\ra$, where the components of~$x$ modulo~$2$ 
run through the set of words of a binary code~$C$. 
The condition 
$\min\Lb=\min\Lb'\,(=1)$ implies that $C$ is of weight $w\ge 4$ 
(because index~$2$ is impossible in dimensions $1,2,3$) 
and that the scalar products $e_i\cdot e_j$ must be zero 
whenever $i,j$ belong to the support of some word of weight~$4$ 
(because the centered cubic lattice $L\sim\D_4$ is the only 
$4$-dimensional lattice with $\imath(L)=2$). 
Under these conditions, the averaging argument of Remark~\ref{rem:averaging}
shows that we may choose $\Lb'=\Z^n$. 
We denote the unique lattice $\Lb$ obtained in this way by $\Lb_C$. 
Its minimal vectors are the $\pm e_i\in\Lb'$ and the vectors 
of the form $\frac{\pm e_i\pm e_j\pm e_k\pm e_\ell}2$ for sets 
$\{i,j,k,\ell\}$ which are the support of a weight~~$4$ word of~$C$. 
For the basic terminology of coding theory used here and in the sequel, we refer to \cite{cs-1998}.

\begin{remark}\label{remLC}
Let $C\ne 0$. Then $\Lb_C$ is not integral, 
and the smallest minimum 
which makes it integral is $2$ if $C$ is even and the intersections 
of the support of its words are even sets; then $\Lb_C$ scaled to 
minimum~$2$ is even \ifff $C$ is doubly even, and $4$ otherwise; 
when scaled to minimum~$4$, $\Lb_C$ is even \ifff $C$~is.
\end{remark}

\begin{prop}\label{propcomplete1} 
Let $C$ be a binary code of weight $w\ge 4$. Denote by $w_4$ 
the number of its weight-$4$ words and by $t$ the number of sets 
$\{i,j\}$ such that $i$ and $j$ do not belong to the support 
of the same weight-$4$ word. Then 
$$s(\Lb_C)=n+8\,w_4\ \nd\ r(\Lb_C)=\frac{n(n+1)}2-t\,.$$ 
\end{prop} 

\begin{proof} 
Write $\Lb_C$ as a union $\cup_w\Lb'+\frac{x_w}2$ where $x_w$ 
runs through a set of representatives with components $0,1$ 
of the words of~$C$. 
It is clear that the minimum of $\Lb'+x$ is equal to 
$\frac{\text{weight}(w)}4$, which gives the result for~$s$. 

For $\eps>0$ small enough, let $\cV_\eps$ be the set 
of lattices of the form $\la e_1,\dots,e_n\ra$ 
with $\abs{N(e_i)-1}\le\eps$ 
($1\le i\le n$) and $\abs{e_j\cdot e_k}\le \eps$ ($1\le j<k\le n$).
Then the set $\cV_\eps$ is a neighborhood of~$\Lb'$ 
in the set $\cE_n$ of similarity classes of well-rounded lattices 
of minimum~$1$. 
The set of lattices obtained from lattices $\Lb'\in\cV_\eps$ 
by adjunctions of the vectors $\frac{x_w}2$ as above also is 
a neighborhood of~$\Lb$ in $\cE_n$, 
and these neighbor lattices will not have minimum~$1$ 
unless $e_i\cdot e_j=0$, whenever $i,j$ lie in the support 
of a same weight-$4$ word. When this condition holds, 
one has $S(L)=S(\Lb)$ for all $L\in\cV_\eps$ 
and small enough~$\eps$. 
This proves that lattices in $\cV_\eps$ depend 
up to similarity on $t$ independent parameters, which shows that the perfection 
co-rank of every $L\in\cV_\eps$ is equal to~$t$. 
\end{proof}

\begin{defi}\label{defcomplete}{\rm 
We say that a binary code $C$ of weight~$4$ is {\em complete} 
if every $2$-set $\{i,j\}$ belongs to at least one weight-$4$ 
word of~$C$.
}
\end{defi}

\begin{remark}\label{rem:eutcomplete}
Lattices generated by $\Z^n$, together with vectors of the form
$v=\frac{e_{i_1}+\dots+e_{i_k}}2$ with $k\ge 4$ are eutactic. Indeed
vectors with $k\ge 5$ may be disregarded. When $k=4$, we have the relation
$$\sum p_{\frac{e_{i_1}\pm\dots\pm e_{i_4}}2}=2(p_{e_{i_1}}+\dots+p_{e_{i_4}}).$$
Thus we may modify the basic eutaxy relation $\sum p_{e_i}=\Id$
of $\Z^n$ by inserting a coefficient $1-\eps$ in front of
$p_{e_{i_1}}+\dots+p_{e_{i_4}}$ and adding $\frac{\eps}2$ times
the eight terms sum above; we conclude the proof by induction on the
number of vectors with~$k=4$, after having chosen a small
enough~$\eps$.

As a consequence, lattices constructed with a complete binary code
are not merely perfect, but also extreme by a theorem of Voronoi.
\end{remark}

Note that we avoid the tempting notion 
of a {\em perfect code} as it already exists in coding theory.
As was pointed out to us by Gilles Z\'emor, complete codes are
the codes having covering radius~$2$. He suggested also 
the following first examples to us:
Extend the definition of a complete code to codes of weight~$3$. 
Then:
\begin{enumerate} 
\item
the extension of such a code by the parity check is a complete code of weight~$4$, and
\item 
the Hamming code with parameters $(2^{p}-1,p,3)$ 
is complete as a weight-$3$ code, so that the extended code 
with parameters $(2^p,p,4)$ is a complete code of weight~$4$.
\end{enumerate}
 
Applied to the Hamming code $\cH_7$, we obtain the extended 
Hamming code $\cH_8$, for which $\Lb_C\simeq\E_8$. 
The easy proposition below 
gives a method for constructing complete codes.

\begin{prop}\label{propcomplete2} 
Let $C$ be a binary $(n,k,4)$ code, and let $C'$ be its extension 
to length $n+2$ by the vector $(0^{n-2},1^4)$. 
Then $C'$ is an 
$(n+2,k+1,4)$ code, and $C'$ is complete \ifff 
$C$ is complete and for every $i<n-1$, there exists $j\ne i$ 
such that $(i,j,n-1,n)$ is the support of a weight-$4$ word of~$C$; 
in particular, $n$ must be even. 
\end{prop} 

We omit the proof, as it is not essential for our classification. 
Applying the proposition, 
we see that the $(2m,m-1,4)$ code generated by the words 
$(1^4,0^{2m-4})$, $(0^2,1^4,0^{2m-6})$,\dots, $(0^{2m-4},1^4)$ 
is complete; the corresponding lattice $\Lb_C$ is isometric 
to $\D_{2m}$. Together with $\cH_8$ and its $(7,3,4)$ subcode, 
this exhausts the list of complete codes of length $\ell\le 8$.


%
%
%

\subsection{2-elementary quotients: classification}%
\label{subsec2-elem}

The classification of lattices $\Lb_C$ up to dimension $n=9$ 
amounts to that of binary codes of weight $w\ge 4$ and length 
$\ell\le 9$. 

For type $2^2$, we state the result for all dimensions.
Binary codes~$C$ of dimension~$2$ contain 
$3$~non-zero words $c_1,c_2,c_3$ of weights $w_1,w_2,w_3\ge 4$, 
and since $\Lb_C$ must non-trivially extend a lattice of lower 
dimension, the supports of two words must cover the set 
$\{1,\dots,n\}$. Codes are described by a basis $c_1,c_2$, 
which may be assumed to satisfy 
$4\le w_1\le w_2\le w_3=2n-w_1-w_2$. It is then easily checked 
that quotients of type~$2^2$ come from $2$-dimensional codes 
which are classified by the rules 
$$
4\le w_1\le \frac{2n}3\ \nd\ 
\max(w_1,4)\le w_2\le n-\frac{w_1}2\,
.
$$ 
Here is the list of weight systems for codes of dimension~$2$ and
length $\ell\le 9$, from which we can read $s(\Lb_C)$ 
of the corresponding lattices $\Lb_C$ and find with little effort 
their perfection rank (using Proposition~\ref{propcomplete1}). 

\begin{itemize}

\item[$n=6$:] $(4^3)$. 

\item[$n=7$:] $(4^2,6)$, $(4,5^2)$. 

\item[$n=8$:] $(4^2,8)$, $(4,5,7)$, $(4,6^2)$, $(5^2,6)$. 

\item[$n=9$:] $(4,5,9)$, $(4,6,8)$, $(4,7^2)$, $(5,5,8)$, $(5,6,7)$, $(6^3)$. 

\end{itemize}

In Table~\ref{tab:n9d22-table} we give a list 
with generators and the associated basic invariants for the $n=9$ cases.

\begin{table}[ht]
\begin{tabular}{|c|c|c|c|}
\hline
generators  & $s$  & $r$ &  $s'$  \\
\hline
\hline
(1,1,1,1,0,0,0,0,0),(0,0,0,0,1,1,1,1,1)  &  $17$  &  $15$ &  $9$ \\
(1,1,1,1,0,0,0,0,0),(0,0,0,1,1,1,1,1,1)  &  $17$  &  $15$ &  $9$ \\
(1,1,1,1,0,0,0,0,0),(0,0,1,1,1,1,1,1,1)  &  $17$  &  $15$ &  $9$ \\
(1,1,1,1,1,0,0,0,0),(0,0,0,0,1,1,1,1,1)  &  $9$   &  $9$  &  $9$ \\
(1,1,1,1,1,0,0,0,0),(0,0,0,1,1,1,1,1,1)  &  $9$   &  $9$  &  $9$ \\
(1,1,1,1,1,1,0,0,0),(0,0,0,1,1,1,1,1,1)  &  $9$   &  $9$  &  $9$ \\
\hline
\end{tabular}
\medskip
\caption{Non-cyclic cases for $n=9$ and $d=2^2$}
\label{tab:n9d22-table}
\end{table}

It is easily checked that the unique code of length~$8$ and weight~$5$ 
has no extension of dimension~$3$ and weight~$5$ to length~$9$.
We may thus from now on restrict ourselves to codes of weight~$4$, 
when considering binary codes of dimension at least~$3$.

\smallskip

\begin{table}
\begin{tabular}{|c|c|c|c|}
\hline
generators aside of (1,1,1,1,0,0,0,0,0) & $s$  &  $r$  &  $s'$  \\
\hline
\hline
(0,0,1,1,1,1,0,0,0),(0,0,0,0,0,1,1,1,1)  &  $41$  &  $30$  & $9$ \\
(0,0,1,1,1,1,0,0,0),(0,0,0,0,1,1,1,1,1)  &  $33$  &  $24$  & $9$ \\
(0,0,1,1,1,1,0,0,0),(0,0,0,1,0,1,1,1,1)  &  $33$  &  $24$  & $9$ \\
(0,0,1,1,1,1,0,0,0),(0,1,0,1,0,1,1,1,1)  &  $33$  &  $24$  & $9$ \\
(0,0,0,1,1,1,1,0,0),(0,0,1,0,0,0,1,1,1)  &  $33$  &  $27$  & $9$ \\
(0,0,0,1,1,1,1,0,0),(0,0,1,0,0,1,1,1,1)  &  $25$  &  $21$  & $9$ \\
(0,0,1,1,1,1,1,0,0),(0,0,0,0,0,1,1,1,1)  &  $25$  &  $21$  & $9$ \\
(0,0,1,1,1,1,1,0,0),(0,1,0,1,0,0,1,1,1)  &  $17$  &  $15$  & $9$ \\
\hline
\end{tabular}
\medskip
\caption{Non-cyclic cases for $n=9$ and $d=2^3$}
\label{tab:n9d222-table}
\end{table}

Next we turn to binary codes of dimension~$3$. 
As shown in~\cite{martinet-2001}, there is one code if $n=7$, 
and three new codes if $n=8$. 
From the list of codes of dimension~$2$ of length $\ell\le 8$ above, 
one easily proves that there are eight new codes in dimension~$9$.
See Table~\ref{tab:n9d222-table}. The basic invariants can be
easily computed using Proposition~\ref{propcomplete1}.

\smallskip

\begin{table}
\begin{tabular}{|c|c|c|c|}
\hline
generators aside of (1,1,1,1,0,0,0,0,0),(0,0,1,1,1,1,0,0,0) &  $s$  &  $r$  &  $s'$  \\
\hline
\hline
(0,1,0,1,0,1,1,0,0),(1,1,0,0,0,0,0,1,1)  &  $89$  &  $43$  &  $9$  \\
(0,1,0,1,0,1,1,0,0),(1,1,0,0,0,0,1,1,1)  &  $65$  &  $30$  &  $9$  \\
(0,0,0,0,1,1,1,1,0),(0,1,0,1,0,1,0,1,1)  &  $57$  &  $37$  &  $9$  \\
(0,0,0,1,0,1,1,1,0),(1,0,1,0,0,0,0,1,1)  &  $81$  &  $45$  &  $9$  \\
\hline
\end{tabular}
\medskip
\caption{Non-cyclic cases for $n=9$ and $d=2^4$}
\label{tab:n9d2222-table}
\end{table}

Extending the four codes of dimension~$3$ and length~$\ell\le 8$, 
we prove that there are four codes in dimension~$9$ besides 
the trivial extension of the $(8,4,4)$ extended Hamming code~$\cH_8$. 
See Table~\ref{tab:n9d2222-table}.
Again, the basic invariants can be
easily computed using Proposition~\ref{propcomplete1}.

\smallskip 

Since the automorphism of $\cH_8$ is $3$-fold transitive 
on the coordinates, it does not extend to a code of dimension~$5$ 
and length~$9$, which completes the classification of $2$-elementary 
codes for $n=9$.
Note that the latter code extends to a  
code of dimension~$5$ and length~$10$ though, 
which lifts to the lattice $\la\E_8,\D_{10}\ra$; 
see Appendix~A.

\subsection{3-elementary quotients}\label{subse3-elem}

Quotients of $\Lb/\Lb'$ of type $3^k$ are construc\-ted using ternary 
codes of weight $w\ge 6$ and dimension~$k$, but the existence 
of a code $C$ does not imply the existence of a pair $(\Lb,\Lb')$ 
defining~$C$, as shown by the lemma and the comment below.

\begin{lemma}\label{lem333} 
There does not exist $9$-dimensional pairs $(\Lb,\Lb')$ 
with $\Lb/\Lb'$ $3$-elementary of order~$27$. 
\end{lemma} 

\begin{proof} 
A ternary code $C$ of length~$9$ and dimension~$3$ extends 
a ternary code $C_0$ of length~$8$ and dimension~$2$. 
There is a unique code $C_0$, and the lattice $\Lb_0$ defined 
by $C_0$ is the $\E_8$ lattice. Hence $\Lb$ must contain to index~$3$ 
a lattice having an $\E_8$ cross-section, 
which contradicts Lemma~\ref{lemsecE8}.
\end{proof}

Note that despite this Lemma, there exists a (unique) ternary code 
with parameters $(9,3,6)$, given by the generating matrix
$$\left(\stx1&1&1&1&1&1&0&0&0 \\ 0&0&-1&-1&1&1&1&1&0 \\ 
0&1&1&-1&-1&0&1&0&1\estx\right) .$$ 
Its weight system is $(9,6^{12})$. 

\smallskip 

It is easily checked that there are three ternary codes 
with parameters~$(9,2,6)$ and one with parameters~$(8,2,6)$. 
Their respective weight systems are $(6^3,9)$, $(6^2,7,8)$, $(6,7^3)$ 
and $(6^4)$; generating matrices for the first three can be read in 
Table~\ref{tab:n9d33-table}; the latter one, 
referred to in Lemma~\ref{lem333}, extends a code of length~$8$.

The averaging argument of Remark~\ref{rem:averaging} 
applied to the first three codes produces 
matrices depending on two, zero, and three parameters. 
In the first and third case, we find lattices in this way 
for which $s(\Lb)$ takes the smallest possible value 
compatible with Watson's conditions 
(see Proposition~\ref{propwatid}). 
Hence the minimal classes 
of our three lattices are the smallest possible, 
with invariants $s,r$ and $s'$ as displayed in Table~\ref{tab:n9d33-table}.

\begin{table}
\begin{tabular}{|c|c|c|c|}
\hline
generators  & $s$  &  $r$  &  $s'$  \\
\hline
\hline
(1,1,1,1,1,1,0,0,0),(0,0,0,1,1,1,1,1,1)  &  $27$  &  $23$  &  $9$ \\
(1,1,1,1,1,1,0,0,0),(1,1,2,0,0,0,1,1,1)  &  $50$  &  $37$  &  $10$ \\
(1,1,1,1,1,1,0,0,0),(1,1,0,0,2,2,1,1,1)  &  $15$  &  $14$  &  $9$ \\
\hline
\end{tabular}
\medskip
\caption{Non-cyclic cases for $n=9$ and $d=3^2$}
\label{tab:n9d33-table}
\end{table}

Below we give Gram matrices for the three lattices $\Lb$, 
obtained by replacing $e_1$ and $e_9$ in a basis $(e_1,\dots,e_9)$ 
for $\Lb'$, by vectors with denominators $3$ and with there 
numerators containing representatives obtained from two code words in Table~\ref{tab:n9d33-table}: 
{\small 
$$ 
\left(\stx94&47&47&47&47&47&0&0&47\\47&90&18&5&5&5&-5&-5&0\\47&18&90&5&5&5&-5&-5&0\\47&5&5&90&18&18&5&5&47\\47&5&5&18&90&18&5&5&47\\47&5&5&18&18&90&5&5&47\\0&-5&-5&5&5&5&90&18&47\\0&-5&-5&5&5&5&18&90&47\\47&0&0&47&47&47&47&47&94\estx\right),
\left(\stx18&9&9&9&9&9&-3&-3&0\\9&18&3&3&0&0&-2&-2&-3\\9&3&18&3&0&0&-2&-2&-3\\9&3&3&18&0&0&-3&-3&-9\\9&0&0&0&18&9&0&0&9\\9&0&0&0&9&18&0&0&9\\-3&-2&-2&-3&0&0&18&3&9\\-3&-2&-2&-3&0&0&3&18&9\\0&-3&-3&-9&9&9&9&9&18\estx\right), 
$$
{}
$$
\left(\stx120&60&60&60&60&60&0&0&0\\60&108&9&9&9&9&0&0&0\\60&9&108&36&9&9&0&0&42\\60&9&36&108&9&9&0&0&42\\60&9&9&9&108&36&0&0&-42\\60&9&9&9&36&108&0&0&-42\\0&0&0&0&0&0&108&27&54\\0&0&0&0&0&0&27&108&54\\0&0&42&42&-42&-42&54&54&110\estx\right) 
\,.$$ 
}

Index systems for the Gram matrices, hence also for the corresponding minimal classes are:

\ctl{$\{1,2,3,4,2^2,6,3^2\},\ \{1,2,3,4,2^2,5,6,7,8,4\cdot 2,9,3^2\},\ \{3,6,2^2\}\, .$}

\subsection{Quotients of type 4$\cdot$2}\label{subse4.2} 

Here, $n=9$. 
We define integers $a_i,b_i,\,1\le i\le n$, $m_1$, $m_2$, 
such that $m_1\ge 4$, $m:=m_1+m_2\in\{7,8,9\}$,  
writing $\Lb$ in the form $\Lb=\la\Lb',e,f\ra$, where 
$$e=\frac{a_1 e_1+\dots+a_n e_n}4\ \nd\ 
f=\frac{b_1\,e_1+\dots+b_n\,e_n}2$$ 
with $a_i\in\{0,1,2\}$, $b_i\in\{0,1\}$, 
$a_i=1$ for $i\le m_1$, $a_i=2$ for $m_1 + 1\le i\le m_1+m_2$, 
$a_i=1$ for $i>m$. 
We also consider 
$$e'=\frac{a'_1 e_1+\dots+a'_n e_n}4\ \nd\ 
f'=\frac{b'_1\,e_1+\dots+b'_n\,e_n}2\,,$$ 
$e'\equiv e+f\mod\Lb'$, $f'\equiv 2e+f\mod\Lb'$, 
$a'_i=\pm 1$ for $i\le m_1$, $a'_i=0\text{ or }2$ 
for $i>m_1$, and $b'_i=0\text{ or }1$. 
Note that $m_1$, namely the number of components $\pm 1$ of words 
attached to denominator~$4$, is an invariant for all codes 
of the form $4\cdot 2^k$. 

We first prove that $m_1=9$ is impossible. 
This shows that minimizing $m_1+m_2$ by exchanging $e$ and $e'$ 
if needed, we may assume  that $m_1+m_2\le 8$, 
i.e., that all codes extend some $7$- or $8$-dimensional
$\Z/4\Z$-code. Then we must have $b_9=1$. 

\smallskip

All together, there are $26$ new codes in dimension~$9$ displayed 
in Table~\ref{tab:n9d42-table} (thus with the extensions of the three 
$8$-dimensional codes, there exist $29$ codes). 
They have been classified by first choosing~$m$, then~$m_1$ 
as small as possible, then choosing~$f$ as short as possible. 
The numbers of codes for given pairs~$(m_1,m_2)$ as above are 

\smallskip 

\ctl{
$(4,3)$: $6$; $(5,2)$: $6$; $(6,1)$: $5$; 
$(4,4)$, $(5,3)$: $1$; $(6,2)$, $(7,1)$: $3$; $(8,0)$: $1$ .} 

\smallskip 

\noi 
They define only $22$~minimal classes, as the two codes with 
$(s,r)=(41,30)$, those with $(s,r)=(33,27)$ in lines $2$ and $4$ 
of Table~\ref{tab:n9d42-table}, and the three codes with $(s,r)=(9,9)$ 
define the same minimal classes. 

\smallskip

\begin{table}
\begin{tabular}{|c|c|c|c|}
\hline
generators  &  $s$  &  $r$  &  $s'$  \\
\hline
\hline
(1,1,1,1,2,2,2,0,0),(0,0,0,0,0,1,1,1,1]) &  $41$  &  $30$  &  $9$  \\
(1,1,1,1,2,2,2,0,0),(0,0,0,1,0,1,1,1,1)  &  $33$  &  $27$  &  $9$   \\
(1,1,1,1,2,2,2,0,0),(0,0,1,1,0,1,1,1,1)  &  $33$  &  $27$  &  $9$   \\
(1,1,1,1,2,2,2,0,0),(0,0,0,1,0,0,1,1,1)  & $33$ &  $27$  & $9$ \\ 
(1,1,1,1,2,2,2,0,0),(0,0,1,1,0,0,1,1,1)  &  $25$  &  $21$  &  $9$   \\
(1,1,1,1,2,2,2,0,0),(0,0,1,1,0,0,0,1,1) & $41$ & $30$ &  $9$ \\ 
(1,1,1,1,1,2,2,0,0),(0,0,0,0,0,1,1,1,1)  &  $17$  &  $15$  &  $9$   \\
(1,1,1,1,1,2,2,0,0),(0,0,0,0,1,1,1,1,1)  &  $23$  &  $22$  &  $9$   \\
(1,1,1,1,1,2,2,0,0),(0,0,0,1,1,1,1,1,1)  &  $56$  &  $37$  &  $12$   \\
(1,1,1,1,1,2,2,0,0),(0,0,0,0,1,0,1,1,1)  &  $17$  &  $15$  &  $9$   \\
(1,1,1,1,1,2,2,0,0),(0,0,0,1,1,0,1,1,1)  &  $9$  &  $9$  &  $9$   \\
(1,1,1,1,1,2,2,0,0),(0,0,0,1,1,0,0,1,1)  &  $24$  &  $21$  &  $9$   \\
(1,1,1,1,1,1,2,0,0),(0,0,0,0,1,1,0,1,1)  &  $46$  &  $34$  &  $9$   \\
(1,1,1,1,1,1,2,0,0),(0,0,0,1,1,1,0,1,1)  &  $23$  &  $21$  &  $9$   \\
(1,1,1,1,1,1,2,0,0),(0,0,0,0,0,1,1,1,1)  &  $35$  &  $28$  &  $9$   \\ 
(1,1,1,1,1,1,2,0,0),(0,0,0,0,1,1,1,1,1)  &  $42$  &  $34$  &  $10$   \\
(1,1,1,1,1,1,2,0,0),(0,0,0,1,1,1,1,1,1)  &  $23$  &  $21$  &  $9$   \\
(1,1,1,1,2,2,2,2,0),(0,0,1,1,0,0,0,1,1)  &  $33$  &  $24$  &  $9$   \\
(1,1,1,1,1,2,2,2,0),(0,0,1,1,0,0,0,1,1)  &  $17$  &  $15$  &  $9$   \\
(1,1,1,1,1,1,2,2,0),(0,0,0,0,1,1,0,1,1)  &  $17$  &  $15$  &  $9$   \\
(1,1,1,1,1,1,2,2,0),(0,0,0,1,1,1,0,1,1)  &  $9$  &  $9$  &  $9$   \\
(1,1,1,1,1,1,2,2,0),(0,0,0,1,1,1,0,0,1)  &  $38$  &  $29$  &  $9$   \\
(1,1,1,1,1,1,1,2,0),(0,0,0,0,1,1,1,0,1)  &  $37$  &  $32$  &  $9$   \\
(1,1,1,1,1,1,1,2,0),(0,0,0,0,0,1,1,1,1)  &  $41$  &  $35$  &  $9$   \\
(1,1,1,1,1,1,1,2,0),(0,0,0,0,1,1,1,1,1)  &  $9$  &  $9$  &  $9$   \\
(1,1,1,1,1,1,1,1,0),(0,0,0,0,1,1,1,1,1)  &  $32$  &  $29$  &  $9$   \\
\hline
\end{tabular}
\medskip
\caption{Non-cyclic cases for~$n=9$ and~$d=4\cdot 2$}
\label{tab:n9d42-table}
\end{table}

\subsection{Quotients of type 4$\cdot$2${}^\text{2}$}%
\label{subse4.2.2}

Here we may write $\Lb=\la\Lb',f,f',f''\ra$ where $f,f',f''$ have 
denominators $4$, $2$, and $2$. By the results above for quotients 
of type $4\cdot 2$, we know that we may assume that $f$ has 
component zero on $e_9$, and then replacing $f'$ by $f''$ or $f'+f''$, 
that $f'$ also have the same property. Hence the $8$-dimensional 
section $\Lb_0=\la e_1.\dots,e_8,f,f'\ra$ is of one of three types, 
characterized by $(m_1,m_2)=(4,3)$, $(5,2)$ or $(6,1)$ 
(see~\cite[Section~10 and Table~11.1]{martinet-2001}).

The third type defines only $\E_8$, and thus does not extend 
to a quotient $(4,2)$ in dimension~$9$ by Lemma~\ref{lemsecE8}. 
For the second one, 
the eutactic lattice in its minimal class is the $8$-dimensional 
{\em Watson lattice}, that is, the unique
integral lattice of minimum~$4$, with $s=75$.
(We refer to it as the ``Watson lattice'', 
as Watson proved that for $n=8$, either $s=120$, attained only 
by the root lattice $\E_8$, or $s\leq 75$; see~\cite{watson-1971b}).
Watson's lattice is the unique weakly eutactic lattice in its minimal
class. It has determinant~$512$, which implies 
$\g(\Lb)\ge2=\g(\Lb_9)$ by Proposition~\ref{prop:sumorth}. 
Thus conjecturally, $\Lb$ is similar to~$\Lb_9$, 
and indeed, we do find only one code, hence only the class of~$\Lb_9$. 
Finally, we find three codes extending the first type. 
Two of them again define the minimal class with $(s,r)=(89,43)$ 
already found for quotients of type $2^4$ and $4\cdot 2$, 
and one the perfect class with $s=81$, already found for quotients 
of type $2^4$. 
The Ryshkov polyhedron of the class with $(s,r)=(89,43)$ 
is a square with edges belonging to a same minimal class having 
$(s,r)=(90,44)$ and vertices belonging to the minimal class of~$\Lb_9$.
Our findings are subsumed in Table~\ref{tab:n9d422-table}.

\begin{table}
\begin{tabular}{|c|c|c|c|}
\hline
generators  &  $s$  &  $r$  &  $s'$  \\
\hline
\hline
(1,1,1,1,2,2,2,0,0),(0,0,1,1,0,0,1,1,0),(0,0,1,1,0,1,0,0,1)  &  $89$  &  $43$  &  $9$  \\
(1,1,1,1,2,2,2,0,0),(0,0,1,1,0,0,1,1,0),(0,1,0,1,0,1,0,0,1)  &  $81$  &  $45$  &  $9$  \\
(1,1,1,1,2,2,2,0,0),(0,0,1,1,0,0,1,1,0),(0,1,0,1,0,0,1,0,1)  &  $89$  &  $43$  &  $9$  \\
(1,1,1,1,1,2,2,0,0),(1,1,0,0,0,1,0,1,0),(0,0,0,0,0,1,1,1,1)  &  $136$  &  $45$  &  $12$  \\
\hline
\end{tabular}
\medskip \caption{Non-cyclic cases for $n=9$ and $d=4\cdot 2^2$}
\label{tab:n9d422-table}
\end{table}

\subsection{Computer calculations} 
\label{subsec:computer}

As mentioned at the beginning of Section~\ref{sec:noncyclic},
the full classification for~$n=9$ relies on computer calculations, 
using an implementation of Algorithm~\ref{alg:index-algorithm}.
In order to keep the necessary computations 
as low as possible, we used a program to systematically generate a list 
of possible cases. It uses the classification of codes for cyclic quotients for~$n=9$.
Note that for a type $d_1 \cdots d_k$ 
(with the $d_i$ having a common divisor greater~$1$)
to be realizable in dimension~$n$, 
all of the $k$ types $d_1 \cdots d_{i-1}\cdot d_{i+1} \cdots d_n$
have to be realizable.

For the cases to be treated, it suffices to consider $k=2$,
say types $d_1 \cdot d_2$, with cyclic types $d_1$ and $d_2$
both existing. 
We can run through all combinations of possible codes generated 
by $a= (a_1,\cdots,a_n) \in (\Z/d_1\Z)$ and 
by $b= (b_1,\cdots,b_n) \in (\Z/d_2\Z)$.
From our classification of cyclic cases, the 
$a_i$ and $b_i$ are assumed to be in $\{0,\ldots, \lfloor \frac{d_1}{2} \rfloor\}$,
respectively $\{0,\ldots, \pm \lfloor \frac{d_2}{2} \rfloor\}$.
This is due to the fact that we could exchange $e_i$ and $-e_i$ in a basis.
For one of the given vectors, say $a$, we may assume that this
property holds; we may moreover assume that the $a_i$ are in non-decreasing order.
For the $b_i$ however, we cannot make this assumption,
as we already used possible sign changes and changes of order of the vectors $e_i$ 
for ``the normalization'' of $a$. 
With each $b_i$, we therefore need to consider also $d_2-b_i$ 
(except when $b_i=0$ or $b_i = \frac{d_2}{2}$ and $d_2$ is even).
Moreover, we need to consider all orderings of the $b_i$ --
up to some symmetry within equal $a_i$ entries. For example,  
if $a_{i}= a_{i+1}\ldots = a_{i+l}$, we may assume
that $b_i,\ldots, b_{i+l}$ are in non-decreasing order.  
Note that we may still be faced with quite a lot of possibilities,
depending on the given choice of $a$ and $b$. 
Note also, that it may be advisable to change the 
roles of~$a$ and~$b$.
Another possibility to reduce the number of cases to be considered: 
For each case we can consider 
linear combinations $f = x\frac{a_1 e_1+\dots+a_n e_n}{d_1} + y\frac{b_1 e_1+\dots+b_n e_n}{d_2}$ 
with $x,y\in\Z$ and check (based on the classification of cyclic types)
if a corresponding lattice $\Lb = \la \Lb', f \ra$ could exist.

\smallskip

Using the strategy sketched above, 
we were able to exclude the types
$8\cdot 2$, $10\cdot 2$, $6\cdot 3$, $9\cdot 3$ and~$5 \cdot 5$
in dimension~$9$.
For the types~$6\cdot 2$ and~$4^2$ we were able to show existence.
We moreover obtained a complete classification of corresponding codes.
See Tables~\ref{tab:n9d62-table} and~\ref{tab:n9d44-table}.
From the classification of $6\cdot 2$, we can exclude 
the last remaining type $6\cdot 2^2$, as we explain below.
Our results were obtained using an implementation of Algorithm~\ref{alg:index-algorithm},
using {\tt MAGMA} scripts in conjunction with {\tt lrs}.
Our source code can be obtained from the online appendix of this paper,
contained in the source files of its arXiv version {\tt arXiv:0904.3110}.
We used a {\tt C++} program that systematically generated 
a list of possible cases as sketched above.

\begin{table}
\begin{tabular}{|c|c|c|c|}
\hline
generators  &  $s$  &  $r$  &  $s'$  \\
\hline
\hline
(0,1,1,1,1,2,2,2,3),(1,0,0,0,1,0,0,1,1)  &  $136$  &  $45$  &  $37$  \\
(0,1,1,1,1,2,2,2,3),(1,0,0,1,1,0,0,1,0)  &  $136$  &  $45$  &  $46$  \\
(0,1,1,1,1,2,2,3,3),(1,0,0,1,1,0,1,0,1)  &   $99$  &  $45$  &  $33$  \\
(0,1,1,1,1,2,2,3,3),(1,0,0,1,1,1,1,0,1)  &   $99$  &  $45$  &  $33$  \\
(0,1,1,1,1,2,2,3,3),(1,0,0,1,1,0,0,0,1)  &   $87$  &  $42$  &  $23$  \\
(0,1,1,2,2,2,2,3,3),(1,0,1,0,0,0,1,0,1)  &   $72$  &  $35$  &  $22$  \\
(0,1,1,1,2,2,2,2,3),(1,0,0,1,0,0,0,1,1)  &   $64$  &  $40$  &  $33$  \\
(0,1,1,1,2,2,2,2,3),(1,0,1,1,0,0,0,1,0)  &   $64$  &  $40$  &  $33$  \\
(0,1,1,1,2,2,2,3,3),(1,0,0,1,0,0,1,0,1)  &   $41$  &  $34$  &  $23$  \\
\hline
\end{tabular}
\medskip \caption{Non-cyclic cases for $n=9$ and $d= 6\cdot 2$}
\label{tab:n9d62-table}
\end{table}

For codes of type $d=6\cdot 2$, the computer assisted calculation 
output was $23$ codes, which we had to check for equivalence. 
Write $\Lb=\la\Lb',e,f\ra$ with 
$$e=\frac{a_1 e_1+\dots+a_9 e_9}6\ \nd\ f=\frac{b_1 e_1+\dots+b_9 e_9}2\,.$$ 
with $a_i\in\{0,1,2,3\}$ and $b_i\in\{0,1\}$. 
Replacing $e$ by $e+f$ or $2e+f$, we obtain 
(after reduction modulo~$6$ and sign changes of some $e_i$) 
three sets $(t_i)$ where $t_i$ is the number of $a_j$ equal to~$i$ 
in the numerator of~$e$. Two equivalent $\Z/6\Z$-codes must have the 
same sets $(t_i)$. For codes having the same sets $(t_i)$, 
we were able to make a canonical choice of an $e'$ among 
$e$, $e+f$ and $2e+f$, constructing this way two new $\Z/6\Z$-codes 
(if one of them is defined by a pair $(e',f')$, the other one 
corresponds to $(e',3e'+f')$). Given two pairs $(e',f')$ 
and $(e',f'')$, we checked whether a convenient permutation 
of the coordinates could transform $f''$ or $f''+3e'$ into $e'$. 
The result is that the $23$ codes found by the computer were 
classified up to equivalence by their three sets $(a,b,c,d)$, 
which reduced our list to only $9$ classes of codes. 

In this list, there are three pairs of lattices having the same 
kissing number ($s=136$, $s=99$, $s=64$). In each case, the matrices 
found by the computer define lattices which are isometric, 
thus defining the same minimal class. 

Quotients of type $6\cdot 2^2$ can be easily ruled out by the 
classification of type $6\cdot 2$. 
Writing now $\Lb=\la\Lb',e,f,f'\ra$, we see on Table~\ref{tab:n9d62-table} 
that we may choose $e$ such that $a_1=0$. Then replacing if necessary 
$f$ by $f'$ or $f+f'$, we may assume that $b_1=0$. But this implies 
the existence of an $8$-dimensional lattice having a quotient 
of type $6\cdot 2$, a contradiction.

\begin{table}
\begin{tabular}{|c|c|c|c|}
\hline
generators  &  $s$  &  $r$  &  $s'$  \\
\hline
\hline
(1,1,1,1,2,2,2,0,0),(0,1,-1,2,2,0,1,2,1) & 81 & 45 & 9 \\
\hline
\end{tabular}
\medskip
\caption{Non-cyclic case for $n=9$ and $d=4^2$}
\label{tab:n9d44-table}
\end{table}

The $4^2$ case is very special. There is only one lattice~$L_{81}$
in the minimal class, which therefore is perfect. 
A Gram matrix is for example
$$
{\small
\begin{pmatrix}
4&1&1&1&2&2&2&0&2\\ 1&4&0&0&0&0&0&0&1\\ 1&0&4&0&0&0&0&0&-1\\ 1&0&0&4&0&0&0&0&2\\ 2&0&0&0&4&0&0&0&2\\ 2&0&0&0&0&4&0&0&0\\ 2&0&0&0&0&0&4&0&1\\ 0&0&0&0&0&0&0&4&2\\ 2&1&-1&2&2&0&1&2&4\\
\end{pmatrix}
}      
$$
The lattice $L_{81}$ and its {\em dual lattice} 
are strongly eutactic, that is, their sets of minimal vectors 
are {\em spherical $3$-designs} (see \cite{martinet-2003}).
Since it is perfect, it is also extreme and dual-extreme by Voronoi's theorem.
The complete index system of~$L_{81}$ is described in Appendix~B.

%
%
%

\section{Universal lattices}\label{secuniversal} 

In this section, we consider lattices $\Lb$ which are {\em universal} 
(for their dimension~$n$) in the following sense: 
with our usual notation, every quotient $L/L'$ which exists 
in dimension~$n$ exists with $L=\Lb$ for a convenient choice 
of $\Lb'$ generated by minimal vectors of~$\Lb$.

\begin{theorem}\label{thuniv} 
For dimensions $n=1,\ldots, 9$, the universal lattices in the sense 
above are as follows (as usual up to similarity): 
\begin{enumerate} 
\item 
All lattices if $n=1$, $2$ or $3$. 
\item 
The root lattice $\D_4$ if $n=4$. 
\item 
All lattices $L$ with $\imath(L)=2$ and $s(L)\ge 6$ if $n=5$. 
\item 
None if $n=6$ or $n=9$. 
\item 
The lattices $\E_7$, $\E_8$ if $n=7,8$. 
\end{enumerate}\end{theorem} 

\begin{proof} 
\underbar{$n\le 3$}. There is nothing to prove since 
the maximal index is~$1$. 

\smallskip\noi\underbar{$n=4$}. 
The maximal index is equal to~$2$ 
only for $\D_4$, which also admits index~$1$ since $\D_4$ has 
bases of minimal vectors. 

\smallskip\noi\underbar{$n=5$}. The maximal index is again~$2$ 
in dimension~$5$, so that we must have $\imath(L)=2$. 
This implies that $L$ may be written in the form 
\linebreak 
$L=L'\cup (f+L')$ 
where $f=\frac{e_1+\dots+e_\ell}2$, $\ell=4$ or~$5$, 
and $L'=\la e_1,\dots,e_5\ra$ has index~$2$ in~$L$. 

If $\ell=4$, then both the conditions ``$s\ge 6$'' and 
``$1$ is an index'' are satisfied. 

If $\ell=s(L)=5$, then $S(L)=S(L')$ and $2$ is the only index 
for~$L$. If $\ell=5$ and $s(L)\ge 6$, there exists some minimal vector 
$f\ne \pm e_i$. If $f\in L'$, then $\ell<5$. Hence $f$ belongs 
to $f+L'$, and is of the form $f=\frac{a_1 e_1+\dots+a_5 e_5}2$. 
We have $\abs{a_i}\le 2$ because $\imath(L)\le 2$, 
and $a_i\ne 0, \pm 2$ because $\ell=5$. Hence $(f,e_2,e_3,e_4,e_5)$ 
is a basis for~$L$. 

\smallskip\noi\underbar{$n=6$}. 
The maximal index is~$4$, attained uniquely on $\D_6$. 
This lattice has index system $\{1,2,2^2\}$. Since there exist 
lattices with $\imath=3$ (e.g.~$\E_6$), there is no universal 
lattice in this dimension. 

\smallskip\noi\underbar{$n=7,8$}. 
It results from Table~11.1 of \cite{martinet-2001} that index~$8$ for $n=7$ and 
index~$16$ for~$n=8$ occur only for $\E_n$. Using the classification 
of root systems, it is then easy to list all well-rounded sublattices 
of minimum~$2$ of $\E_7$ and $\E_8$, see \cite{martinet-2001}, Section~6, 
and then to check that they realize all quotients in their 
dimensions. 

\smallskip\noi\underbar{$n=9$}.
Quotients $4^2$ occur only on the similarity class of the perfect lattice~$L_{81}$ 
described at the end of the previous section, whereas cyclic 
quotients of order~$12$ occur only on similarity classes with $s \ge 87$, 
as described in Section~\ref{sec:cyclicd12}. 
See also Table~\ref{tab:Table81_87_99} in Appendix~B.
\end{proof}

\begin{remark}\label{rem:lambda9_almost_universal} 
The lattice $\Lb_9$ is almost universal. It realizes all types in 
dimension~$9$, except $4^2$. 
\end{remark}

\begin{proof} 
We must show that all quotients listed 
in Theorem~\ref{thm:structquot}, except $4^2$, 
do occur as quotients of~$\Lb_9$. 
This is clear for those which belong 
to the index system of~$\E_8$ since $\Lb_9$ has a cross-section 
proportional to~$\E_8$. It thus remains to consider quotients 
which are either cyclic of order $7,8,9,10$ and $12$ or of type 
$6\cdot 2$ or $4\cdot 2^2$. Luckily, this problem can be solved 
by a mere inspection of the codes found for dimension~$9$: 
indeed, for each of these quotients, there exists at least one code 
for which $\Lb_9$ is the only admissible lattice. 
Here is a list of such codes, given with the notation 
of \ref{notacyclic} in cyclic cases 
and by the components of generators otherwise. 

\noi Type $(7)$\,: $(6,1,2)_7$\,; 

\noi Type $(8)$\,: $(4,3,2,0)_8$\,; 

\noi Type $(9)$\,: $(4,1,2,2)_9$\,; 

\noi Type $(10)$\,: $(2,4,2,0,1)_{10}$\,; 

\noi Type $(12)$\,: $(2,1,2,2,1,1)_{12}$\,; 

\noi Type $(6\cdot 2)$\,: 
$(0,1^4,2^3,3)_6,\,(1,0^3,1,0^2,1^2)_2$\,; 

\noi Type $(4\cdot2^2)$\,: 
$(1^5,2^2,0^2)_4,\,(1^2,0^3,1,0,1,0)_2,\,(1^2,0^4,1,0,1)_2$\,. 
\end{proof}

We do not know any result of this kind for larger dimensions. 
Note that Remark~\ref{remsmallindex} shows that a $24$-dimensional 
universal lattice, if any, must be the Leech lattice. 
In dimension~$10$, a possible universal lattice is provided 
by the lattice $\la\E_8,\D_{10}\ra$, which has quotients of order~$32$ 
and of the three types $2^5$, $4\cdot 2^3$ and $4^2\cdot 2$; it has a cross-section 
$\Lb_9$, but we do not even know whether all quotients of $\Lb_{10}$ 
occur for this lattice.

%
%

\begin{appendix}
 
\section*{Appendix A: Some perfect lattices} 

As usual, the notation $\A_n$, $\D_n$, $\E_n,\,n=6,7,8$ 
stands for the standard irreducible {\em root lattices}, 
the definitions of which we recall below. 
Their importance stems from Witt's theorem, 
which asserts that {\em integral 
\linebreak 
lattices generated by vectors of norm~$2$ 
are orthogonal sums of lattices isometric to $\A_n,\,n\ge 1$, 
$\D_n,\,n\ge 4$, or $\E_n,\,n=6,7,8$.} 
Denoting by $(\vp_0,\vp_1,\dots,\vp_n)$ the canonical basis 
for $\Z^{n+1}$ and by $(\vp_1,\dots,\vp_n)$ that of $\Z^n$, 
we set 
$$\A_n=\left\{x\in\Z^{n+1}\mid \sum_{i=0}^{n+1}\,x_i=0\right\}\ 
\nd\ 
\D_n=\left\{x\in\Z^n\mid\sum_{i=1}^n\,x_i\equiv 0\mod 2\right\}$$ 
(we consider $\A_n$ for $n\ge 1$ and $\D_n$ for $n\ge 2$, 
but $\D_2\simeq\A_1\perp\A_1$ and $\D_3\simeq\A_3$). 
For all $n\ge 8$ even, we then set 
$$\D_n^+=\la\D_n,\frac{\vp_1+\dots+\vp_n}2\ra= 
\D_n\cup \left(\frac{\vp_1+\dots+\vp_n}2+\D_n \right)\,,$$ 
and $\E_8=\D_8^+$ (but $\D_n^+$ is not a root lattice for $n>8$), 
and finally define $\E_7$ and $\E_6$ as the orthogonal complement in~$\E_8$ 
of the spans of $\vp_7+\vp_8$ and $\{\vp_6+\vp_7,\vp_7+\vp_8\}$ respectively. 

Note that $\A_n$ has a nice characterization in terms of its index 
system, by a 1877 theorem of Korkine and Zolotareff: 
it is {\em the} $n$-dimensional lattice with $s\ge\frac{n(n+1)}2$ 
and maximal index~$1$. For $\D_n$, we quote the following property: 

\begin{Appprop}\label{propsysDn} 
For all $n\ge 2$, the index system of $\D_n$ is 
$$\cI(\D_n)=\{1,2,\dots,2^t\}\ ,$$ where $t=\lf\frac{n-2}2\rf$. 
\end{Appprop} 

\begin{proof}[Sketch of proof] 
By the classification of root systems, 
a strict sublattice of $\D_n$ of rank~$n$ is an orthogonal sum 
of irreducible root lattices $L_1,\dots,L_k$ of dimensions 
$n_1,\dots,n_k<n$ which add to~$n$. 
Embeddings $\E_m\hookrightarrow\D_n$ are impossible (see \cite[Section~4.6]{martinet-2002}).
For $m\ne 1,3$, embeddings $\A_m\hookrightarrow\D_n$ are equivalent 
modulo an automorphism of $\D_n$ to 
$\A_m\to L=\la\vp_1-\vp_2,\dots,\vp_m-\vp_{m+1}\ra$, 
and must be discarded because $L^\perp$ is not a root 
sublattice of~$\D_n$. For $m=3$, there is a second orbit, 
namely that of $\A_3\to L=\la\vp_1-\vp_2,\vp_2-\vp_3,\vp_1+\vp_3 \ra$, 
for which we have $L^\perp\simeq\D_{n-3}$; we denote by $\D_3$ 
this kind of embedding of~$\A_3$. Finally, $\A_1\perp\A_1$ embeds 
as $\la\vp_i\pm\vp_j\ra$, which yields an orthogonal decomposition 
$\D_2\perp\D_{n-2}$ where we denote by $\D_2$ any $\A_1\perp\A_1$ 
embedded as $\la\vp_i+\vp_j, \vp_i-\vp_j\ra$. With these 
definitions of $\D_2$ and $\D_3$, we prove inductively that root 
sublattices of $\D_n$ are obtained taking $L_i=\D_i$, 
and the proof of the proposition is now easily completed. 
\end{proof}

\smallskip 

The {\em laminated lattices $\Lb_n$} were defined inductively 
by Conway and Sloane; see \cite[Chapter~6]{cs-1998}. 
They have minimum~$4$, they are integral in the range $1\le n\le 24$, 
uniquely defined except for $n=11,12,13$, and for $n\le 8$, 
they are scaled copies of 
$\A_1$, $\A_2$, $\A_3$, $\D_4$, $\D_5$, $\E_6$, $\E_7$, $\E_8$. 

\smallskip 

For all even $m\ge 8$ and all $n\ge m$, 
the lattices $\la\D_m^+,\D_n\ra$ (Barnes's lattices $\D_{n,m}$; 
see \cite[Section~5.5]{martinet-2003}) have minimum~$2$. They are integral 
if $m=n\equiv 0\mod 4$, only half-integral otherwise, 
hence become integral in the scale which give them minimum~$4$. 
In particular, 
$$\Lb_9=\la\D_8^+,\D_9\ra =\la\E_8,\D_9\ra 
\text{ scaled to minimum~$4$}\,.$$ 

\smallskip

Here are unified constructions for the three perfect lattices which were
found directly in our classification.
There are four $[10,5,4]$-codes. They lift over $\Z^{10}$ to four
$10$-dimensional lattices:
$L_a\sim\D_{10}^+$ ($s=90$),
$L_b\sim\la\E_8,\D_{10}\ra$ ($s=154$),
$L_c$ ($s=138$, not $K_{10}$), and
$L_d\sim Q_{10}$ ($s=130$, the Souvignier lattice).

Among the densest cross-sections of $L_b$, $L_c$, $L_d$ we find the lattices
$\Lambda_9$, $L_{99}$ and $L_{81}$, respectively.
(For $L_a$ we obtain the lattice with $s=57$ of Table~\ref{tab:n9d2222-table}.)

The three perfect lattices above are indeed eutactic and hence extreme by Voronoi's theorem: 
this is clear for $\Lambda_9$, which contains~$\D_9$ scaled to 
minimum~$4$ and for $L_{81}$, which is strongly eutactic.
For the lattice~$L_{99}$  
we have verified eutaxy using by a computer calculation; 
see Section~\ref{sec:cyclicd12}.

\smallskip

The laminated lattice $\Lb_{24}$ is known as the {\em Leech lattice}
and has many remarkable properties.
By a recent theorem 
of Cohn and Kumar \cite{ck-2004}, we have $\g_{24}=4$ and the only 
lattice attaining $\gamma_{24}$
(up to similarity) is the Leech lattice 
$\Lb_{24}$, an integral lattice of minimum~$4$ and determinant~$1$, 
whence $\g_{24}^{12}=2^{24}$. It follows that $\Lb_{24}$
is the unique lattice in dimension~$24$ that 
satisfies the index bound of Proposition~\ref{propmaxind} with equality:
By the proof of Conway's uniqueness theorem for the Leech lattice
(see \cite[Chapter~12]{cs-1998}),   
every class of $\Lb_{24}\mod 2$ has a representative of norm 
at most $2\min\Lb_{24}=8$, and norm~$8$ vectors occur in $24$ pairs 
$\pm x$, which implies by~\cite[Theorem~2.5]{martinet-2002} 
that $\Lb_{24}$ contains a sublattice $L$ which is a scaled copy 
with minimum~$4$ of $\D_{24}$. We have $\det(L)=4\cdot 2^{24}$, 
hence $[\Lb_{24}:L]=2^{13}$. Now the root lattice $\D_n$, $n=2m$ even, 
contains orthogonal frames, which span lattices of index 
$2^{m-1}$ in~$\D_n$. This shows that $L$ contains to index $2^{11}$ 
a lattice $\Lb'$ generated by minimal vectors of $\Lb_{24}$, 
and we have $[\Lb_{24}:\Lb']=2^{13}\cdot 2^{11}=2^{24}$ 
showing that the bound $\imath(\Lb) \leq \lfloor \g_{24}^{12}\rfloor$ is tight.
This implies the known result 
(see \cite{bcs-1995}) that the Leech lattice 
can be constructed as the pull-back of a code of length~$24$ over $\Z/4\Z$.

%
%

\section*{Appendix B: Enumerating independent subsets of shortest vectors, by Mathieu Dutour Sikiri\'c}

In the present paper, the authors consider 
a pair $(\Lb,\Lb')$ of a $9$-dimensional lattice $\Lb$ 
and one of its Minkowskian sublattices, and classify all 
the $\Z/d\Z$ codes associated with the quotient $\Lb/\Lb'$.
They in particular obtain all possible structures of $\Lb/\Lb'$ 
as an Abelian group. However, given a lattice $\Lb$, 
the question of how to compute its index system~$\cI(\Lb)$ 
(in the sense of Definition~2.2) is left aside. 

\smallskip 

For a given lattice $\Lb$ of dimension~$n$ and (half) kissing 
number~$s$, a straightforward approach would be to consider all of the
$\binom s n$ possible bases of Minkowskian sublattices.
Except for the lattice~${\mathsf E}_8$, this approach works
for all perfect lattices for $n\le 8$, where $s\leq 75$ (see~\cite{martinet-2001}).
In dimension~$9$ however, 
several interesting lattices cannot be handled by this naive approach.
This is in particular true for the interesting lattices $L_{81}$ and $L_{99}$ 
described in Sections~\ref{sec:cyclic} and~\ref{sec:noncyclic}.

\smallskip 

In this appendix, I shall describe shortly an algorithm which outputs 
the index system of some lattices with a large kissing number.
The results I obtained for $L_{81}$ and $L_{99}$ are displayed in Table~\ref{tab:Table81_87_99}.
I also consider the non-perfect lattice $L_{87}$ of perfection rank~$42$ 
and maximal index~$12$  (see Section~\ref{sec:cyclic}).
The minimal class of  $L_{87}$  lies below 
that of $L_{99}$, so that every index which occurs 
for $L_{99}$ already occurs for $L_{87}$.

\smallskip

We denote by $\Aut(\Lb)$ the group of lattice automorphism of $\Lb$.
We split $S(\Lb)$ (the set of minimal vectors of $\Lb$)
into pairs of antipodal vectors $\{v_1, -v_1\}$, \dots, $\{v_s, -v_s\}$ and define $S^{1/2}(\Lb)=\{v_1, \dots, v_s\}$.
The group $\Aut(\Lb)$ induces an action on the $s$ antipodal pairs
and thus defines a permutation group $\Aut^{1/2}(\Lb)$ on $s$ elements of $S^{1/2}(\Lb)$.
If $\Lb$ does not admit a decomposition $\Lb_1\perp \Lb_2$ into two
orthogonal sublattices then the order of $\Aut^{1/2}(\Lb)$ is half the order of $\Aut(\Lb)$.
Denote by $I_k(\Lb)$ the list of inequivalent representatives of orbits of 
independent subsets with $k$ elements of $S^{1/2}(\Lb)$ under $\Aut^{1/2}(\Lb)$.

We need to determine $I_n(\Lb)$, but it turns out that the only known method
requires enumerating $I_k(\Lb)$ for $k\leq n$ as well.
Given $I_k(\Lb)$, for all $S\in I_k(\Lb)$ we consider
all possible ways to add one vector to $S$ and get an independent system.
By keeping only inequivalent representatives, we get in this way $I_{k+1}(\Lb)$.
The basic problem is to be able to test if two subsets of $S^{1/2}(\Lb)$
are equivalent under the group $\Aut^{1/2}(\Lb)$.
There exist backtracking method for this purpose, that are known to work
well in practice (see \cite[Chapter 9]{seress-2003}). Using these techniques, 
we find $|I_8(\mathsf{E}_8)|=1943$ for the highly symmetric $\mathsf{E}_8$ lattice.

The basic problem of this method is that we have to store $I_k(\Lb)$
in memory and that the number of equivalence tests grows quadratically in the
size of $I_k(\Lb)$.
To overcome these difficulties we use an ``orderly generation'' approach, a classic
technique of combinatorial enumeration (see for example~\cite{mckay-1998}).

If $S\in I_n(\Lb)$ then we choose $S$ to be lexicographically minimal
in its orbit under $\Aut^{1/2}(\Lb)$ and write it as $S=\{x_1, \dots, x_n\}$
with $x_1 < x_2 < \dots < x_n$.
Then the
sets $S_k=\{x_1, \dots, x_k\}$ for $1\leq k\leq n$ are lexicographically minimal
in their respective orbits as well.
Reversely, suppose we have all lexicographically minimal representatives
in $I_k(\Lb)$,
then for all $S_k=\{x_1,\dots, x_k\}\in I_k(\Lb)$ we consider all sets
\begin{equation*}
S_k(t)=S_k\cup \{t\}\mbox{~for~} t\in \{x_k+1, \dots, s\}.
\end{equation*}
and we test for all of them if they are minimal in their orbit $O(S_k(t))$
under $\Aut^{1/2}(\Lb)$ by computing all elements of $O(S_k(t))$.
If they are minimal then they are added to the set $I_{k+1}(\Lb)$.
Obviously
this method is limited by the size of the group and is not appropriate for
$\mathsf{E}_8$ or $\Lb_9$.

\smallskip

\begin{table}
\begin{center}
\begin{tabular}{|c||c|c|c|}
\hline
lattice $\Lb$   & $L_{81}$  & $L_{87}$ & $L_{99}$\\\hline
$| \Aut^{1/2}(\Lb) |$ & 18432 & 6144 & 9216\\
\hline
\hline
$1$    & 3774844 & 16730092 & 49301288\\
$2$    & 474881  & 2657720  & 8271400\\
$3$    & 28768   & 198528   & 681759\\
$4$    & 6634    & 46390    & 163090\\
$2^2$  & 4579    & 28560    & 88407\\
$5$    & 348     & 2859     & 12126\\
$6$    & 205     & 2171     & 8462\\
$7$    & 3       & 59       & 230\\
$8$    & 7       & 49       & 169\\
$4,2$  & 57      & 212      & 597\\
$2^3$  & 32      & 132      & 309\\
$9$    & --       & 4        & 4\\
$3^2$  & --       & 1        & 12\\
$10$   & --      & 5        & 5\\
$12$   & --       & 1        & 1\\
$6,2$  & --       & 1        & 10\\
$4^2$  & 1       & --       &  --\\
$4,2^2$& 1       & --        & --\\
$2^4$  & 1       & --        & --\\
\hline
\end{tabular}
\end{center}
\medskip
\caption{The number of orbits of bases of Minkowskian sublattices, 
for each of the~$19$ possible index types in dimension 9 and
for the special lattices~$L_{81}$, $L_{87}$ and~$L_{99}$}
\label{tab:Table81_87_99}
\end{table}

Once the sets $I_n(\Lb)$ are built, we use the Smith Normal Form for
each element (basis of a Minkowskian sublattice $\Lb'$), 
to determine the invariant of the Abelian group $\Lb/\Lb'$.

In our, obviously non-optimal, implementation we store the sets $I_k(\Lb)$
on disk and we use the {\tt GMP} library for exact arithmetic and a {\tt C} program that builds
$I_{k+1}(\Lb)$ from $I_k(\Lb)$. The Smith Normal Form computation is done in
{\tt GAP}. 
The running time is always less than 1 week.
The program is part of the {\tt GAP} package {\tt polyhedral} \cite{polyhedral}.

\end{appendix}

%
%

\section*{Acknowledgments}

The authors thank Gilles Z\'emor for suggesting the example
of a complete code described in Proposition~\ref{propcomplete2}.
They thank Mathieu Dutour Sikiri\'c for contributing Appendix~B, 
and for computing the index systems of some lattices.
They thank Bertrand Meyer for helpful suggestions on the text.
The third author would like to thank the 
Institut de Math\'ematiques at Universit\'e Bordeaux~1
for its great hospitality during two visits,
on which major parts of 
this work were created.

%
%

\providecommand{\bysame}{\leavevmode\hbox to3em{\hrulefill}\thinspace}
\providecommand{\MR}{\relax\ifhmode\unskip\space\fi MR }
\providecommand{\MRhref}[2]{%
  \href{http://www.ams.org/mathscinet-getitem?mr=#1}{#2}
}
\providecommand{\href}[2]{#2}

\end{document}